\documentclass [a4paper,reqno, 11pt]{amsart}
\pdfoutput=1

\usepackage{fullpage}
\usepackage{graphicx} % Required for inserting images
\usepackage{amssymb,amsthm,amsmath,mathrsfs,amsfonts}
\usepackage[thinc]{esdiff}
\usepackage{xcolor}
\usepackage[bookmarks, bookmarksdepth=2, colorlinks=true, linkcolor=blue, citecolor=blue, urlcolor=blue]{hyperref}

\usepackage[font=footnotesize,labelfont=bf,labelsep=colon]{caption}
\usepackage{subcaption}

\usepackage{tikz}
\usepackage{pgfplots}

\usepackage{booktabs}

\usepackage[indent,skip=.2\baselineskip]{parskip}

\theoremstyle{plain}
\newtheorem{theorem}{Theorem}[section]

\theoremstyle{plain}
\newtheorem{corollary}[theorem]{Corollary}

\theoremstyle{plain}
\newtheorem{proposition}[theorem]{Proposition}

\theoremstyle{plain}
\newtheorem{lemma}[theorem]{Lemma}

\theoremstyle{definition}
\newtheorem{definition}[theorem]{Definition}

\theoremstyle{definition}
\newtheorem{example}[theorem]{Example}

\theoremstyle{definition}
\newtheorem{remark}[theorem]{Remark}

\theoremstyle{definition}
\newtheorem{problem}[theorem]{Problem}

\theoremstyle{definition}
\newtheorem{conjecture}[theorem]{Conjecture}

\theoremstyle{definition}

\newcommand{\init}{\mathrm{in}}

\usepackage{url}

\newcommand{\PP}{\mathbb{P}}
\newcommand{\CC}{\mathbb{C}}

\newcommand{\RR}{\mathbb{R}}

\newcommand{\EE}{\mathbb{E}}

\newcommand{\IG}{\IG}

\newcommand{\I}{\mathcal{I}} %ideal

\newcommand{\M}{\mathcal{M}} %moment variety
\newcommand{\K}{\mathcal{K}} %cumulant variety

\DeclareMathOperator{\reg}{reg} %regularity
\DeclareMathOperator{\codim}{codim} %codimension
\DeclareMathOperator{\Sec}{Sec} %secant variety
\DeclareMathOperator{\Ass}{Ass} %associated primes
\DeclareMathOperator{\HS}{HS} %associated primes
\DeclareMathOperator{\height}{ht} %height

\newcommand{\from}{\colon}

\numberwithin{equation}{section}

\title[Moment varieties from inverse Gaussian and gamma distributions]{Moment varieties from inverse Gaussian and\\ gamma distributions}

\author{Oskar Henriksson, Lisa Seccia, Teresa Yu}

\begin{document}

\maketitle

\begin{abstract}
Motivated by previous work on moment varieties for Gaussian distributions and their mixtures, we study moment varieties for two other statistically important two-parameter distributions: the inverse Gaussian and gamma distributions. In particular, we realize the moment varieties as determinantal varieties and find their degrees and singularities. We also provide computational evidence for algebraic identifiability of mixtures, and study the identifiability degree and Euclidean distance degree. 
\end{abstract}

\section{Introduction}

Suppose $X$ is a univariate random variable from a distribution that depends on finitely many parameters $\theta=(\theta_1,\ldots,\theta_n)$, and we want to determine $\theta$ from observed data. One approach to this problem is the \emph{method of moments}. 

The $r$th \emph{moment} of $X$ is defined as the expected value $m_r(\theta):=\mathbb{E}(X^r)$.
For many distributions, the moments $m_r(\theta)$ are polynomials in the parameters $\theta$ (see, e.g., Table~2 of \cite{BS15}). This makes it possible to estimate $\theta$ by computing the sample moments $\widetilde{m}_r:=\frac{1}{N}\sum_{i=1}^N x_i^r$ for some large number $N$ many observations, and then solving the system of polynomials
\begin{equation}\label{eq:sample_moments}
    m_r(\theta)=\widetilde{m}_r,\quad r\in[d]
\end{equation}
for some $d\in\mathbb{N}$. The law of large numbers implies that we can expect arbitrarily good approximations of $\theta$ among the solutions of this system, if $N$ is large enough. In particular, the method of moments gives rise to a consistent estimator in many cases  (see, e.g., \cite[Theorem~9.6]{Was04}).

For each $d\in\mathbb{N}$, one can define the \emph{$d$th moment variety} $\M_d\subseteq\PP^d$ as the Zariski closure of the image of the map
\[\CC^n\to\mathbb{P}^d,\quad \theta\mapsto [m_0(\theta):m_1(\theta):\cdots:m_d(\theta)].\]
This family of projective varieties provides a starting point for the algebraic-geometric perspective on the method of moments. We are particularly interested in the notions of \emph{algebraic} and \emph{rational identifiability} of the parameters, which correspond to understanding when the system \eqref{eq:sample_moments} has finitely many complex solutions or a unique solution, and therefore to understanding when the parameters $\theta$ can be recovered from the moments. Such notions of identifiability can be interpreted as understanding generic fibers of the map $\CC^n\to\mathbb{P}^d$, and can therefore be studied with techniques from algebra and geometry.

An important tool for studying moment varieties is the concept of \emph{cumulants} and their associated varieties $\K_d$ (see Definition~\ref{defn:cumulant}). There is an isomorphism of affine varieties $\M_d\cap\{m_0\neq 0\}\cong\K_d$, and passing from moments to cumulants often gives rise to simpler varieties (see, e.g., \cite[\S2]{AFS16} and the more general discussion in \cite[\S4--5]{CCMRZ16}). Cumulants also have the additional advantage of being additive over independent variables. However, moments behave simpler with respect to certain other statistical operations; in particular, moment varieties of mixtures are secant varieties, which is a fact that we will return to in \S\ref{sec:future_directions}. Furthermore, the cumulant variety only captures an open dense patch of the moment variety, and therefore might miss important geometric aspects of the moment variety such as singularities.

\subsection*{Previous work on moment varieties}
The algebraic study of moments and cumulant varieties goes back to Pearson in 1894 \cite{Pea94}, who studied the mixtures of two univariate Gaussians. The case of Gaussians and Gaussian mixtures was later revisited from a nonlinear algebra perspective in \cite{AFS16}, as well as in the subsequent papers \cite{ARS18,AAR21,LAR21} that addressed various identifiability questions, and (in the latter case) approximation methods based on numerical algebraic geometry. In \cite{AFS16}, the authors study moments as projective varieties. The main advantage of studying moments from the perspective of projective geometry is that it allows for the study of secant varieties, which correspond to the moment varieties for mixtures of distributions. In general, algebraic and even rational identifiability for distributions listed in \cite[Table~2]{BS15} is clear, but it is unknown for mixtures of these distributions.  Other distributions for which moment and cumulant varieties have been studied include uniform distributions on polytopes \cite{KSS20}, Dirac and Pareto distributions \cite{GKW20}, and mixtures of products \cite{AKS23}.

\subsection*{Main contributions}
In this paper, we study the moment varieties for the inverse Gaussian and gamma distributions, which are both two-parameter distributions used in a wide array of applications. See Figure~\ref{fig:pdf} for examples of density functions, and see Figure~\ref{fig:moment_varieties} for an illustration of the moment varieties for $d=3$.

\begin{figure}
\begin{subfigure}{0.3\textwidth}
\centering
\begin{tikzpicture}[
scale=0.6,
declare function={pdf_gauss(\x,\mu,\v)=1/(sqrt(\v*2*pi))*exp(-1/2*((\x-\mu)/sqrt(\v))^2);}
]
\begin{axis}[
axis lines=left,
enlargelimits=upper,
samples=100,
column sep=0.2em,
row sep=0.2em,
legend cell align={left},
legend entries={
$\mu=-2{,}\:\:\: \sigma^2=5$,
$\mu=-1{,}\:\:\: \sigma^2=5$, 
$\mu=-2{,}\:\:\: \sigma^2=2$}
]
\addplot [smooth, domain=-6:4.5, red]{pdf_gauss(x,-2,5)};
\addplot [smooth, domain=-6:4.5, blue]{pdf_gauss(x,-1,5)};
\addplot [smooth, domain=-6:4.5, teal]{pdf_gauss(x,-2,2)};
\end{axis}
\end{tikzpicture}
\caption{Gaussian distribution}
\end{subfigure}
~
\begin{subfigure}{0.3\textwidth}
\centering
\begin{tikzpicture}[
scale=0.6,
declare function={pdf_ig(\x,\mu,\lambda)=sqrt(\lambda/(2*pi*\x^3)*exp(-\lambda*(\x-\mu)^2/(2*\mu^2*\x));}
]
\begin{axis}[
axis lines=left,
enlargelimits=upper,
samples=100,
column sep=0.2em,
row sep=0.2em,
legend cell align={left},
legend entries={
$\mu=1{,}\:\:\: \lambda=5$,
$\mu=2{,}\:\:\: \lambda=5$, 
$\mu=9{,}\:\:\: \lambda=2$}
]
\addplot [smooth, domain=0:5, red]{pdf_ig(x,1,5)};
\addplot [smooth, domain=0:5, blue]{pdf_ig(x,2,5)};
\addplot [smooth, domain=0:5, teal]{pdf_ig(x,1,2)};
\end{axis}
\end{tikzpicture}
\caption{Inverse Gaussian distribution}
\end{subfigure}
~
\begin{subfigure}{0.3\textwidth}
\centering
\begin{tikzpicture}[
scale=0.6,
declare function={gamma_function(\z)=(2.506628274631*sqrt(1/\z) + 0.20888568*(1/\z)^(1.5) + 0.00870357*(1/\z)^(2.5) - (174.2106599*(1/\z)^(3.5))/25920 - (715.6423511*(1/\z)^(4.5))/1244160)*exp((-ln(1/\z)-1)*\z);},
declare function={pdf_gamma(\x,\k,\theta) = \x^(\k-1)*exp(-\x/\theta)/(\theta^\k*gamma_function(\k));}
]
\begin{axis}[
axis lines=left,
enlargelimits=upper,
samples=100,
column sep=0.2em,
legend cell align={left},
legend entries={
$k=1{,}\:\:\: \theta=1$,
$k=2{,}\:\:\: \theta=1$, 
$k=1{,}\:\:\: \theta=2$}
]
\addplot [smooth, domain=0:5, red]{pdf_gamma(x,1,1)};
\addplot [smooth, domain=0:5, blue]{pdf_gamma(x,2,1)};
\addplot [smooth, domain=0:5, teal]{pdf_gamma(x,1,2)};
\end{axis}
\end{tikzpicture}
\caption{Gamma distribution}
\end{subfigure}
\caption{Probability density functions for three distributions.}
\label{fig:pdf}
\end{figure}
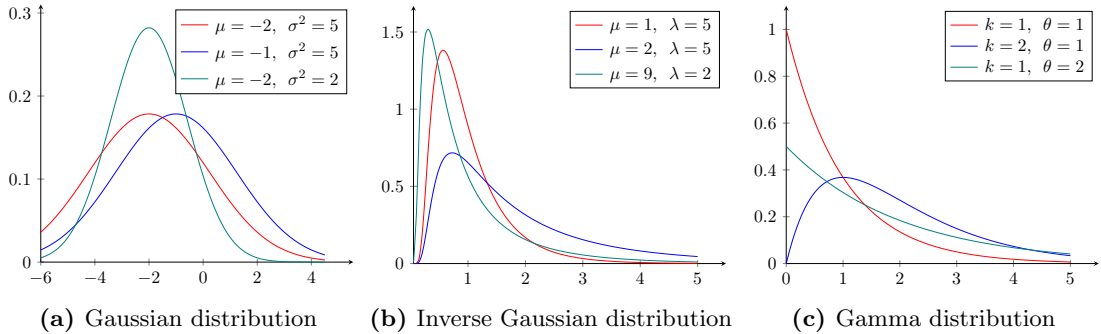

Our main results (Theorems~\ref{thm:invgauss ideal} and \ref{thm:gamma ideal}) give the homogeneous prime ideals for these moment varieties. We show that these ideals are \emph{determinantal ideals}, in the sense that they are generated by all maximal minors of a matrix. Interestingly, both these varieties turn out to be Cohen--Macaulay, which raises Problem \ref{Q:CM} on Cohen--Macaulayness of moment varieties.

Using these results, we find the degrees and Hilbert series for these moment varieties, and give Gröbner bases for their ideals. In particular, Corollary~\ref{prop:gamma_gb_hs} addresses a conjecture stated in \cite{aThesis} on the Hilbert series for the moment variety of the gamma distribution. We also identify the singular loci of the moment varieties (Propositions~\ref{prop:singular_locus_inverse_gaussian} and \ref{prop:singular_locus_gamma}). 

Additionally, we study the exponential and chi-squared distributions, which are one-parameter specializations of the gamma distribution, and we show that their moment varieties are rational normal curves (Propositions~\ref{prop:generators_for_exponetial} and \ref{prop:ideal chi-square}). 

Using the moments-to-cumulants transformation, we prove that the cumulant varieties of the inverse Gaussian are (scaled) Veronese varieties (Proposition~\ref{prop:cumulant_variety_for_inverse_guassian}); this has previously been shown to be the case for the gamma distribution \cite[Proposition~3.2.1]{aThesis}. These simple geometric models, together with previous results on the projective geometry of the corresponding moment varieties, could be the starting point for addressing the question of identifiability of mixtures.

We conclude the paper by providing computational evidence in \S\ref{subsec:identifiability} towards a conjecture on algebraic identifiability for mixtures of inverse Gaussian distributions and mixtures of gamma distributions (Conjecture \ref{conj:nondefec}). We also report on numerical estimations of the identifiability degree (\S\ref{subsec: ID}), and compute some Euclidean distance degrees (\S\ref{subsec: EDD}). 

\begin{figure}
    \centering
    \begin{subfigure}{0.30\textwidth}
    \centering
    \includegraphics[width=0.7\textwidth]{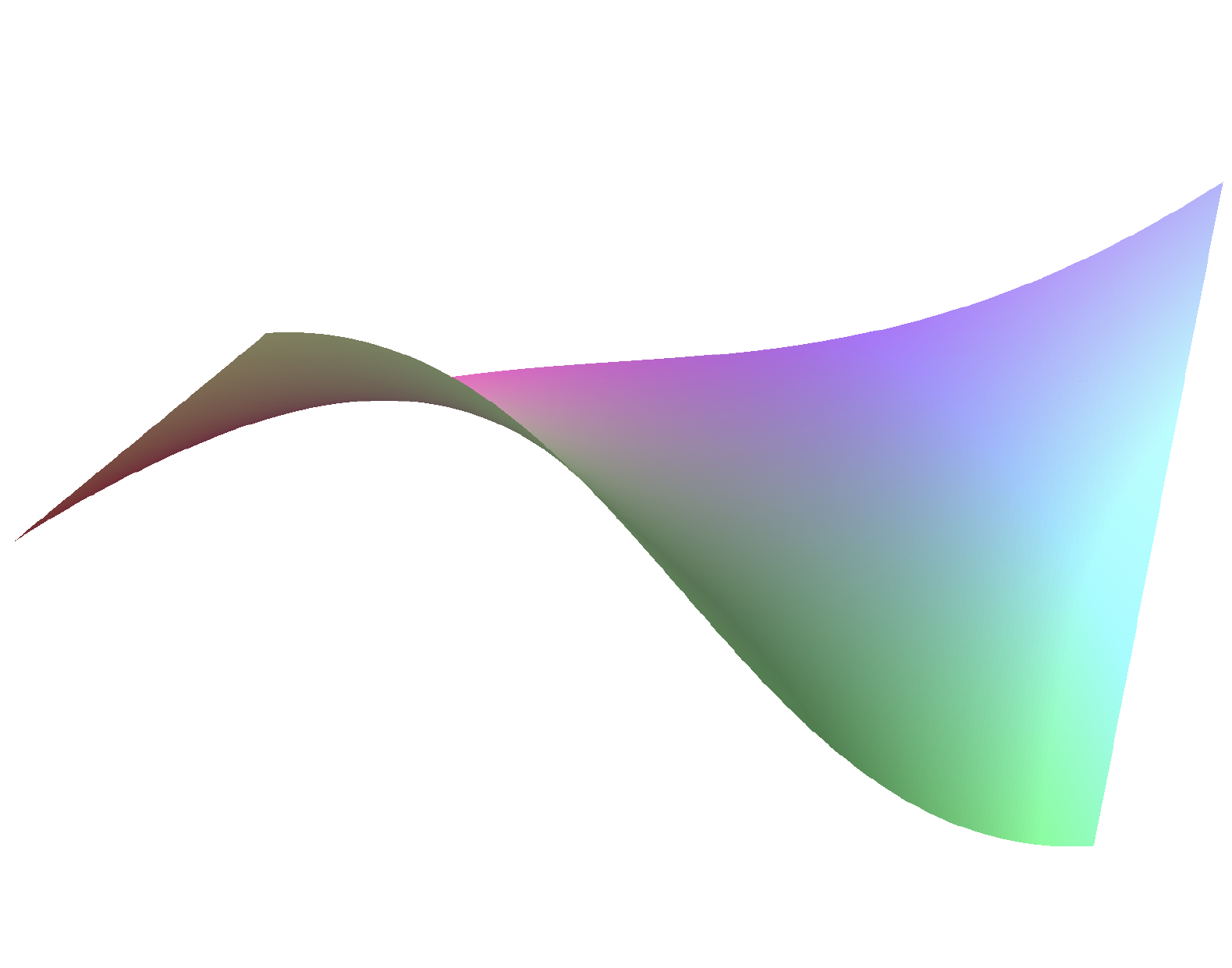}
    \vspace{0.5em}
    \caption{Gaussian distribution}
    \end{subfigure}
    \hspace{0.01\textwidth}
    \begin{subfigure}{0.30\textwidth}
    \centering
    \includegraphics[width=0.7\textwidth]{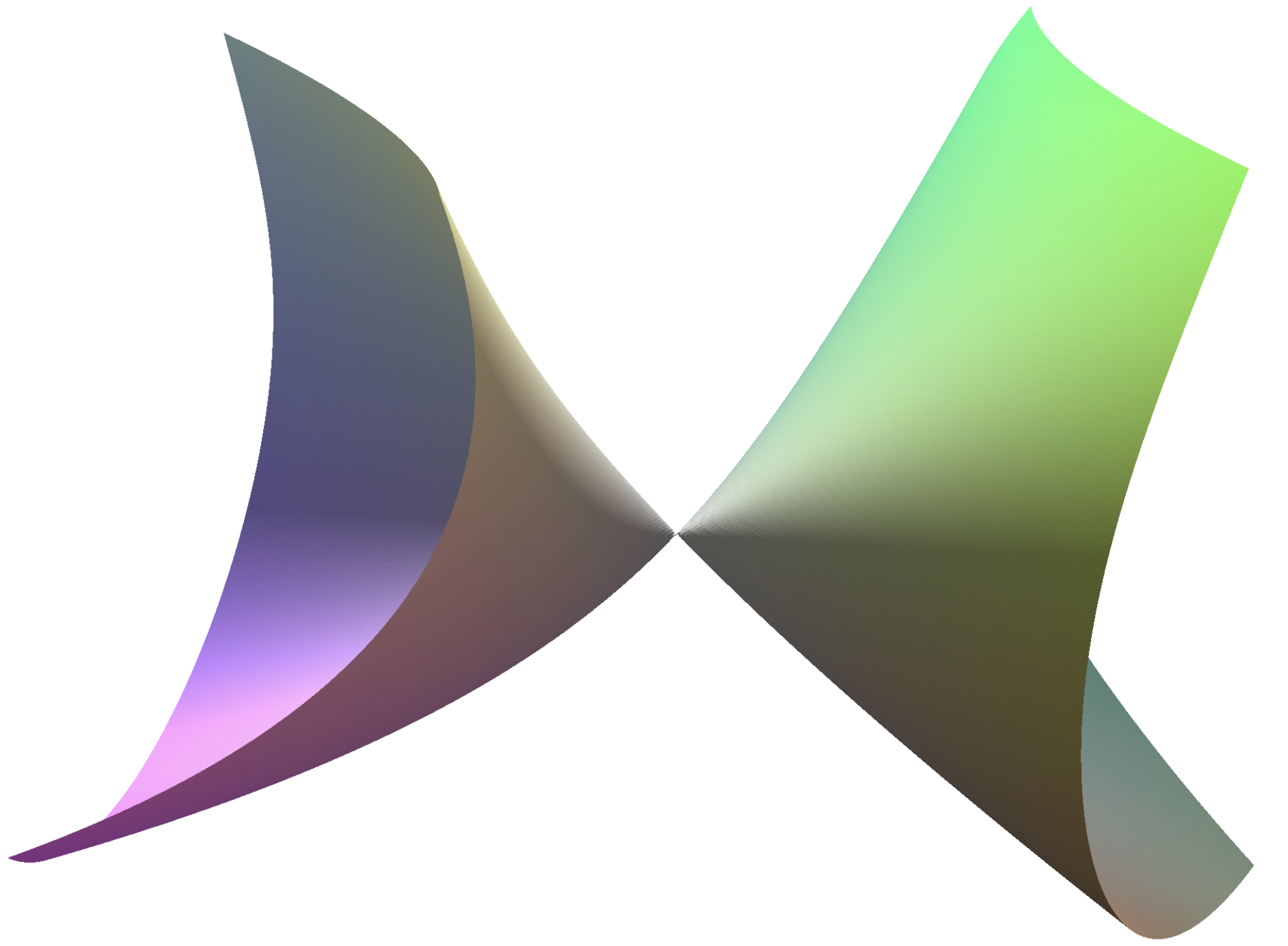}
    \vspace{0.5em}
    \caption{Inverse Gaussian distribution}
    \end{subfigure}
    \hspace{0.01\textwidth}
    \begin{subfigure}{0.30\textwidth}
    \centering
    \includegraphics[width=0.7\textwidth]{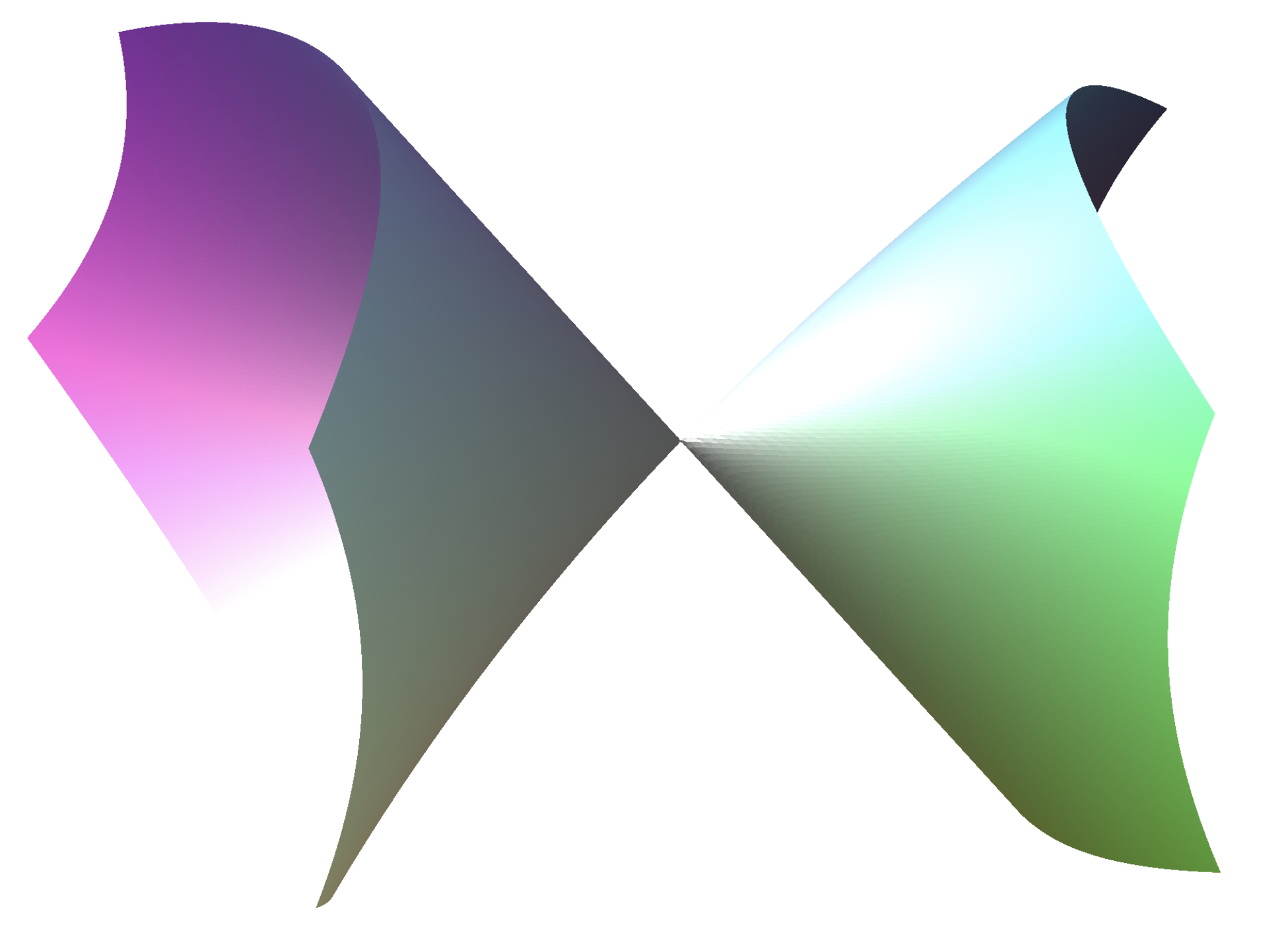}
    \vspace{0.5em}
    \caption{Gamma distribution}
    \end{subfigure}
    \caption{Illustration of the $\{m_0=1\}$ patch of $\M_3$ for three distributions.}
    \label{fig:moment_varieties}
\end{figure}

\subsection*{Applications of commutative algebra techniques}
Moment varieties and other varieties arising in algebraic statistics provide interesting yet concrete algebraic structures to which one can apply techniques from commutative algebra. In this paper, the most important techniques are from the theories of determinantal ideals and Cohen--Macaulay rings, and we provide necessary general background and references in \S\ref{sec:prelim}, as well as further references as needed. We refer the reader to \cite{BCRV,BV,Eis} for further general background on these topics. The Cohen--Macaulay property is particularly important in many of our results. Cohen--Macaulay varieties can be seen as ``mildly singular'' varieties with good homological properties and for which intersection multiplicity behaves well. For instance, they are equidimensional varieties with no embedded components. We also use polarization and Stanley--Reisner theory from combinatorial commutative algebra in \S\ref{sec:gamma}. We hope that our proof techniques inspire further applications of commutative algebra to algebraic statistics.

\subsection*{Outline of the paper}
The rest of the paper is organized as follows. In \S\ref{sec:prelim}, we provide preliminary definitions regarding moment varieties, as well as some background on the commutative algebra results and techniques that we will use. In \S\ref{subsec:inverse_gaussian}, we study moment varieties for the inverse Gaussian distribution. In \S\ref{sec:gamma}, we study moment varieties for the gamma distribution, as well as for certain one-parameter specializations. In \S\ref{sec:future_directions} of the paper, we outline a number of interesting directions for future work on the inverse Gaussian and gamma distributions. 

Code for the computational experiments presented in this paper is publicly available at 

\begin{center}
    \small
    \url{https://github.com/oskarhenriksson/moment-varieties-inverse-gaussian-and-gamma}. 
\end{center}

\subsection*{Notation and conventions}
Throughout the paper, $\mathcal{V}(I)$ denotes the variety associated to a homogeneous ideal $I$; we consider this variety as a projective variety unless otherwise specified. 
For a variety $V$ in a given projective or affine space, we use $\mathcal{I}(V)$ to denote the vanishing ideal.
Finally, for a matrix $H$ with entries in a ring, we use $I_t(H)$ to denote the ideal generated by all $(t\times t$)-minors of $H$.

\subsection*{Acknowledgements}
The authors thank Bernd Sturmfels for suggesting the problem and for his guidance throughout the project, as well as Carlos Am\'endola, Oliver Clarke, Francesco Galuppi, Alexandros Grosdos Koutsoumpelias, Julia Lindberg, Lizzie Pratt, Kristian Ranestad, Jose Israel Rodriguez, and Matteo Varbaro for helpful discussions. Part of this research was performed while the authors were visiting the Institute for Mathematical and Statistical Innovation (IMSI) for the Algebraic Statistics long program, which was supported by the National Science Foundation (NSF grant DMS-1929348), and the authors thank the other participants of the program for their feedback and comments. The authors also thank the Max Plank Institute for Mathematics in the Sciences for hosting the first and third authors and providing a productive visit, during which part of this research was also performed. OH was partially supported by the Novo Nordisk project with grant reference number NNF20OC0065582.  LS was partially supported by SNSF grant TMPFP2-217223. TY was partially supported by NSF grant DGE-2241144.

\section{Preliminaries}\label{sec:prelim}

In this section, we provide some background on moment varieties, as well as preliminaries on the techniques and results from commutative algebra that will be used in the rest of the paper.

\subsection{Moment and cumulant varieties}

Let $X$ be a univariate random variable parameterized by $\theta=(\theta_1,\ldots,\theta_n)$. The \emph{moment generating function} is given by
\[M(t,\theta):=\EE(e^{tX})=\sum_{r=0}^\infty\frac{m_r(\theta)t^r}{r!},\] 
and we recall from the introduction that the \emph{$d$th moment variety} $\M_d\subseteq\PP^d$ is defined as the Zariski closure of the image of the map
\[\CC^n\to\mathbb{P}^d,\quad \theta\mapsto [m_0(\theta):m_1(\theta):\cdots:m_d(\theta)].\]

It is typically the case that if $d$ is large enough, the map $\CC^n\to\mathcal{M}_d$ is generically one-to-one (meaning that the system \eqref{eq:sample_moments} has a unique complex solution for generic $\widetilde{m}\in\M_d$), and we say that we have \emph{rational identifiability}. A weaker condition for the map is \emph{algebraic identifiability}, meaning that it is generically finite-to-one (in the sense that the system \eqref{eq:sample_moments} has finitely many complex solutions for generic $\widetilde{m}\in\M_d$). By the theorem of the dimension of fibers, we have algebraic identifiability if and only if $\dim(\mathcal{M}_{d})=n$. 

For many univariate classical distributions, such as those listed in \cite[Table 2]{BS15}, rational identifiability follows immediately from the polynomials $m_r(\theta)$. However, it is generally much more difficult to know when $k$-mixtures of such distributions are algebraically or rationally identifiable from their moments, and we discuss this further in \S\ref{sec:future_directions}.

\begin{example}
    The univariate Gaussian distribution is parameterized by the mean $\mu$ and variance $\sigma^2$, and its moment generating function is given by
    \[M(t,\mu,\sigma^2)=\sum_{r=0}^\infty\frac{m_r(\mu,\sigma^2)}{r!}t^r=\exp(t\mu)\cdot\exp\left(\frac{1}{2}\sigma^2t^2\right).\]
    The polynomials in $\mu,\sigma^2$ for the moments of order up to $4$ are
    \[
    m_0=1,\quad
    m_1=\mu,\quad
    m_2=\mu^2+\sigma^2,\quad
    m_3=\mu^3+3\mu\sigma^2,\quad
    m_4=\mu^4+6\mu^2\sigma^2+3\sigma^4.
    \]
    These polynomials parameterize the fourth moment variety $\M_4\subseteq\PP^4$ of the univariate Gaussian distribution. By \cite[Proposition 2]{AFS16}, the homogeneous prime ideal of $\M_4$ is the ideal $I\subseteq\CC[m_0,\ldots,m_4]$ generated by the $3\times 3$ minors of the matrix
    \[\begin{pmatrix}
        0 & m_0 & 2m_1 & 3m_2 \\
        m_0 & m_1 & m_2 & m_3  \\
        m_1 & m_2 & m_3 & m_4
    \end{pmatrix}.\]
\end{example}

\begin{definition}\label{defn:cumulant}
    The $r$th \emph{cumulant} $k_r(\theta)$ of the probability distribution $X$ is defined by the cumulant generating function
\[K(t,\theta):=\log(M(t,\theta))=\sum_{r=1}^\infty\frac{k_r(\theta)t^r}{r!},\]
and the $d$th \emph{cumulant variety} $\K_d$ is defined as the Zariski closure of the image of the map
\[\CC^n\to\CC^d,\quad \theta\mapsto (k_1(\theta),\ldots,k_d(\theta)).\]
\end{definition}

The relation $K(t,\theta) = \log(M(t,\theta))$ gives rise to a nonlinear change of coordinates \[\CC^d\to\CC^d,\quad(m_1,\ldots,m_d)\mapsto (m_1, m_2-m_1^2, m_3-3m_2m_1+2m_1^3,\ldots)\]
that expresses the moments of order $\leq d$ as polynomials in the cumulants of order $\leq d$, and vice versa. It restricts to an isomorphism of affine varieties $\M_d\cap\{m_0\neq 0\}\to\K_d$, and hence a birational map $\M_d\dashrightarrow\K_d$, and is an example of a \textit{Cremona transformation} in the language of \cite{CCMRZ16}. For a more extensive background on the statistical properties of cumulants, we refer to \cite[\S2]{McCullagh18}.

\subsection{Background on commutative algebra}
We collect here the main commutative algebra results that will be used throughout the paper. For the rest of this subsection, $\dim(S)$ for a ring $S$ denotes its Krull dimension.

We begin by recalling some fundamental properties of Cohen--Macaulay rings. For a more comprehensive introduction to this class of rings, we refer the reader to \cite[Chapter 2]{BH}. 
Cohen--Macaulay rings are important because they exhibit good homological behavior that simplifies their algebraic and geometric structure. These rings have a well-behaved dimension theory  (they are defined by having depth equal to Krull dimension), which ensures they have the following desirable properties.
\begin{proposition} \label{prop:CMproperties}
    Let $R$ be a Noetherian Cohen--Macaulay ring. Then:
\begin{enumerate}
    \item $R$ is equidimensional, i.e., all the minimal prime ideals of $R$ have same dimension.
    \item $R$ has no embedded primes. In particular, it is unmixed, i.e., all the associated prime ideals of $R$ have same dimension.
    \item $\dim (R/I)=\dim R -\height(I)$ for any ideal $I \subset R$, where $\height(I)$ denotes the height or codimension of $I$.
     \item $r$ is a non-zero-divisor in $R$ if and only if $\dim (R/(r))=\dim R -1$.
\end{enumerate}
\end{proposition}

Cohen--Macaulay rings are ubiquitous in algebra and geometry. They frequently arise as coordinate rings of many important varieties, such as smooth algebraic varieties and certain rings of invariants. 

\begin{example}
The following are some standard examples and non-examples of Cohen--Macaulay rings:
    \begin{itemize}
        \item Regular (local) rings are Cohen--Macaulay. In particular, a standard graded polynomial ring over a field  is Cohen--Macaulay.
        \item Complete intersection rings are Cohen--Macaulay.
        \item Let $R=\CC[x,y]/(x^2y,x)$. Its associated primes are $(x,y)$ and $(x)$. In particular, $R$ has an embedded component, and so condition (ii) above fails. Therefore, $R$ is not Cohen--Macaulay.
        \item Let $R=\CC[x,y,z]/(xy,xz)$. Its associated primes are $(x)$ and $(y,z)$. In this case $R$, has no embedded component, but it is not equidimensional and so condition (i) above fails. Therefore, $R$ is not Cohen--Macaulay.
    \end{itemize}
\end{example}

In this paper, we often work with \emph{determinantal ring}, i.e., rings of the form $R=S/I$ where $S$ is a standard graded polynomial ring and $I=I_t(H)$ is an ideal generated by the $t$-minors of a $k \times \ell$ matrix $H$ with entries in $S$. The following result is key for proving that moment varieties of inverse Gaussian and gamma distributions have determinantal realizations (Theorems \ref{thm:invgauss ideal} and \ref{thm:gamma ideal}).

\begin{proposition}[{\cite[Corollary~3.4.10]{BCRV}}]\label{prop:CMdetIdeals}
    Let $R=S/I$ be a determinantal ring, where $I=I_t(H)$ with $H$ a matrix of size $k\times \ell$. If $I$ has codimension $(k-t+1)(\ell-t+1)$, then $R$ is Cohen--Macaulay.
\end{proposition}

For computational purposes, it is often useful to have a Gröbner basis of the homogeneous ideals defining our moment varieties. In Propositions \ref{prop:invgauss gb} and \ref{prop:gamma_gb_hs}, we prove that the natural generators of these determinantal ideals form Gr\"obner bases with respect to a suitable term order. To do so, we use some standard results from commutative algebra that we collect in the next two lemmas.

\begin{lemma}[{\cite[Proposition 1.4.7]{BCRV}}] \label{lemma: Grobner basis IG}
Let $S$ be a standard graded polynomial ring and let $I \subseteq S$ be a homogeneous ideal. Consider the ideal $J=(\init(f_1), \ldots, \init(f_r))$, where $f_1, \ldots, f_r$ are homogeneous elements of $I$. Then $J =\init(I)$ if the following conditions hold:
\begin{itemize}
    \item $\dim (S/I)= \dim (S/J)$;
    \item $J$ is unmixed;
    \item $\deg(J)=\deg (I)$.
\end{itemize}

\end{lemma}

\begin{lemma}[{\cite[\S15.1.1]{EisCA}}] \label{lemma:SES}
Let $I$ be a monomial ideal and let $M$ be a minimal generator of $I$ of degree $s$. Write $I=I'+(M)$ for an ideal $I'$. Then, there is a short exact sequence of graded modules and degree 0 maps
    \begin{equation}\label{SES}
            0 \longrightarrow S/(I':M)[-s] \longrightarrow S/I' \longrightarrow S/I \longrightarrow 0.
    \end{equation}
Since the Hilbert series is additive in short exact sequences, this shows that
\begin{equation}
        \HS_{S/I}(t)= \HS_{S/I'}(t)- \HS_{S/(I':M)[-s]}(t).
    \end{equation}
\end{lemma}

A special class of determinantal ideals that arise in our study of moment varieties are those coming from $1$-generic (Hankel) matrices (see e.g., Propositions \ref{prop:generators_for_exponetial} and \ref{prop:ideal chi-square}).  
The notion of $1$-generic matrix was introduced by Eisenbud in \cite{Eis}. A matrix $H$ is called \emph{$1$-generic} if it has no generalized entries which are zero, meaning that no entries are identically zero after arbitrary scalar row and column operations. Examples of $1$-generic matrices are generic matrices and Hankel matrices (up to some unit coefficients). The ideals of maximal minors of these matrices are well-understood, and the following result of Eisenbud shows that such ideals are prime and of expected codimension.

\begin{proposition}[{\cite[Theorem 1]{Eis}}]\label{prop:Eis 1-gen}
Let $H$ be a $1$-generic matrix with entries in a ring $S$, and of dimension $k\times \ell$ matrix with $k\leq \ell$. Then
 $I_k(H)$ is prime and of codimension $\ell-k+1$. In particular, $I_k(H)$ is Cohen--Macaulay.
\end{proposition}

\section{Inverse Gaussian moment varieties}\label{subsec:inverse_gaussian}
 
We now study moment varieties of the \emph{inverse Gaussian} distribution, which is also sometimes called the \emph{Wald} distribution. We show that the ideals of its moment varieties are determinantal. Based on this, we use results from the theory of determinantal ideals to find degrees and Gröbner bases, and we compute the singular loci of the moment varieties.

Despite being less known than the classical Gaussian distribution, the inverse Gaussian has significant applications in modeling different phenomena. While the Gaussian distribution is widely known for its symmetry and versatility in modeling various data, the inverse Gaussian distribution offers a valuable alternative for situations where asymmetry and right-skewed tails are more appropriate (see Figure \ref{fig:pdf}), such as in modeling lifetime phenomena. We refer the reader to \cite{Ses93} for further background.

Similar to the Gaussian distribution, the inverse Gaussian is defined by two parameters, namely $\mu$ and $\lambda$. However, its density function is supported on $(0,+\infty)$ and it is given by 
\[ \textstyle f_{\mu,\lambda}(x)={\sqrt {\frac {\lambda }{2\pi x^{3}}}}\,\exp {\biggl (}-{\frac {\lambda (x-\mu )^{2}}{2\mu ^{2}x}}{\biggr )}.\]
Here, $\mu>0$ represents the mean while $ \lambda>0$ is called \emph{shape parameter} since it affects how peaked and right-tailed the density function is. Unlike the Gaussian distribution, $\lambda$ does not directly coincide with the variance, though these two quantities are related by the following formula:
$$\text{Var}(X)=\dfrac{\mu^3}{\lambda}.$$

The main motivation behind the definition of this distribution comes from Brownian motion. In this context, the inverse Gaussian works as a dual for the Gaussian distribution, in the sense that the inverse Gaussian models the first passage time of the Brownian motion to a certain fixed level, whereas the Gaussian models the Brownian motion's level at a fixed time \cite[Ch. 13]{MO}.

The moment generating function of the inverse Gaussian is given by 
\[M(t)=\exp\left(\tfrac{\lambda}{\mu}\left(1-\sqrt{1- \tfrac{2\mu^2t}{\lambda}}\right)\right),\]
and the moments are given by the recursive formula
\begin{equation}\label{eq:parametrization_inverse_gaussian}
    m_0= 1,\quad
    m_1= \mu,\quad 
    m_i= \frac{2i-3}{\lambda} \mu^2 m_{i-1}+\mu^2 m_{i-2}\:\:\text{for}\:\:i\geq 2.
\end{equation}
Although these functions are rational in $\lambda$, one could equivalently parameterize the inverse Gaussian distribution using $1/\lambda$ to obtain moment functions that are polynomial in the parameters for the distribution.

The above system of equations parameterizes the moment variety $\M_d^\mathrm{IG}$ of the inverse Gaussian of degree $d$. Using this parameterization we prove that, for fixed $d$, the moment variety has a determinantal realization and it is Cohen--Macaulay.

\begin{theorem}\label{thm:invgauss ideal}
    Let $d \geq 3$. The homogeneous prime ideal of the inverse Gaussian moment variety $\M^\mathrm{IG}_d$ is generated by   $\binom{d-1}{3}$ cubics and $\binom{d-1}{2}$ quartics, given by the maximal minors of the following $(3\times d)$-matrix
  $$H_d^\mathrm{IG}= 
\begin{pmatrix}
 m_0^2 & m_0 & m_1 & m_2 & m_3&\cdots& m_{d-2}\\
 0 & m_1 & 3m_2 & 5m_3 & 7m_4 &\cdots& (2d-3)m_{d-1}\\
 m_1^2 & m_2 & m_3 &m_4  &m_5& \cdots& m_d
\end{pmatrix} .
$$
Furthermore, $\mathcal{M}^\mathrm{IG}_{d}$ is Cohen--Macaulay.
\end{theorem}
\begin{proof}
The vector $(\mu^2,\mu^2/\lambda,-1)\neq 0$ is in the left kernel of the matrix $H_d^\mathrm{IG}$, so the 3-minors of $H_d^\mathrm{IG}$ vanish on $\M^\mathrm{IG}_d$. Thus, $$J_d:=I_3(H_d^\mathrm{IG}) \subseteq \mathcal{I}(\M^\mathrm{IG}_d).$$
This gives that $\dim(\mathcal{V}(J_d))\geq\dim(\M^\mathrm{IG}_d)=3$ as affine varieties. Let $\operatorname{in}(J_d)$ denote the initial ideal of $J_d$ with respect to the reverse lexicographic ordering; then we have that
\begin{equation}\label{eq:initial_ideal_inverse_gaussian}
    \operatorname{in}(J_d)\supseteq(m_1^4,m_2^3,\ldots,m_{d-2}^3),
\end{equation}
so $\dim(\mathcal{V}(J_d))=\dim(\mathcal{V}(\init(J_d)))\leq d+1-(d-2)= 3$. Thus, we have proved that $$\dim (\mathcal{V}(J_d))= \dim (\M^\mathrm{IG}_d)=2$$ as projective varieties. This shows that the ring $R=S/J_d$, where $S=\CC[m_0,\ldots,m_d]$, has expected Krull dimension $d+1-(d-3+1)=3$, so it is Cohen--Macaulay by Proposition \ref{prop:CMdetIdeals}.
In particular, the ideal $J_d$ has no embedded components by Proposition \ref{prop:CMproperties}.

We want to prove that $J_d$ is prime, so that $J_d= \mathcal{I}(\M^\mathrm{IG}_d)$. This is equivalent to proving that $R$ is a domain. Let $\Delta= m_0^2m_1$ be the $2$-minor in the upper-left corner, and consider the localization $R_{\Delta}$. 
To prove that $R$ is a domain, it is enough to prove that $\Delta$ is a non-zero-divisor in $R$, and that $R_\Delta$ is a domain. 

To prove that $\Delta$ is a non-zero-divisor in $R$, we prove that $m_0$ and $m_1$ are both non-zero-divisors in $R$. Since $R$ is Cohen--Macaulay, by Proposition \ref{prop:CMproperties} it suffices to prove that $\dim (R/(m_0))= \dim (R/(m_1))= \dim(R)-1= 2$. Observe that
\[R/(m_0)\cong \dfrac{\CC[m_1,\ldots,m_d]}{I_3 ((H_d^\mathrm{IG})\vert_{m_0=0})},\]
where 
$$\left(H_d^\mathrm{IG}\right)\vert_{m_0=0}= \begin{pmatrix}
0  & 0 & m_1 & m_2&\cdots& m_{d-2}\\
0 & m_1 & 3m_2 & 5m_3 &\cdots& (2d-3)m_{d-1}\\
m_1^2 & m_2 & m_3 &m_4  & \cdots& m_d
\end{pmatrix}.$$

If we consider the 3-minor on the first three columns, we get $m_1 =0$. Shifting to the next adjacent minor we get $m_2=0$. Thus, iterating, we eventually find that $\mathcal{V}\left(I_3 (H_d^\mathrm{IG}) \right)\cap\{m_0=0\}$ is given by the following projective curve:
\begin{align*}
    m_0&=m_1=\cdots=m_{d-3}=m_{d-2}=0.
\end{align*}
So, $R/(m_0)$ has Krull dimension $2$. A similar argument shows that $\mathcal{V}(I_3(H_d^\mathrm{IG}))\cap\{m_1=0\}$ is a union of two projective curves:
\begin{align*}
    m_0&=m_1=\cdots=m_{d-3}=m_{d-2}=0,\\
    m_1&=m_2=\cdots=m_{d-3}=m_{d-1}=0.
\end{align*}
So, $R/(m_1)$ has Krull dimension $2$. This proves that $\Delta$ is a non-zero-divisor in $R$.

It remains to prove that $R_\Delta$ is a domain. Since $\Delta$ is a non-zero-divisor and $R$ is Cohen--Macaulay, we get that $\dim (R_{\Delta})=\dim (R)=3$. In fact, in order to see that $\dim( R_\Delta)=\dim (R)$, it is sufficient to note that $\Delta$ is not nilpotent. This is because $R$ is a finitely generated $\CC$-algebra and therefore the nilradical of $R$ is equal to the Jacobson radical of $R$. In particular, if $\Delta$ is not nilpotent, then it is not in some maximal ideal $\mathfrak{m}$ of $R$. Since $R$ is Cohen--Macaulay, the length of a maximal chain of prime ideals ending at $\mathfrak{m}$ is $\dim(R)=3$, and since $\Delta\notin\mathfrak{m}$, this chain is preserved in $R_\Delta$. Thus, $\dim(R_\Delta)=\dim(R)=3$. We now claim that $R_\Delta$ is a domain. Note that for $3\leq i\leq d$, one can inductively use the $3$-minor of $H_d^\mathrm{IG}$ on columns $1,2,i$ to see that $\overline{m_i}\in R_\Delta$ can be expressed as
\[\overline{m_i}=\overline{f}/\overline{m_0^2m_1},\qquad \text{for some }f\in\mathbb{C}[m_0,m_1,m_2].\]
Thus, $R_\Delta$ is a finitely generated $\mathbb{C}$-algebra:
\[R_\Delta=\mathbb{C}\left[\overline{m_0},\overline{m_1},\overline{m_2},\overline{1/m_0^2m_1}\right]\subseteq R.\]
In particular, we get an isomorphism $\mathbb{C}[w,x,y,z]/(w^2xz-1)\cong R_\Delta$, induced by the $\CC$-algebra homomorphism
\[\varphi\from \CC[w,z,y,z]\to R_\Delta,\quad w\mapsto\overline{m_0},\quad x\mapsto\overline{m_1},\quad y\mapsto\overline{m_2},\quad z\mapsto\overline{1/m_0^2m_1}.\]
If $(w^2xz-1)\subsetneq\ker(\varphi)$, then $\dim \mathcal{V}(\ker(\varphi))\leq 2$ since $(w^2xz-1)$ is prime. This is because the height of $\ker(\varphi)$ would need to be strictly greater than the height of $(w^2xz-1)$, and $\CC[w,x,y,z]$ is Cohen--Macaulay so $\dim \mathcal{V}(I)+\text{ht}(I)=4$ for every ideal $I$. However, $\dim(R_\Delta)=3$. Thus it must be that $\ker(\varphi)=(w^2xz-1)$, and therefore $R_\Delta$ is a domain, which completes the proof.
\end{proof}

As an immediate consequence of the previous result, we get two corollaries on important invariants of $\mathcal{I}(\M^\mathrm{IG}_d)$. These corollaries concern the theory of free resolutions, Eagon--Northcott complexes, and (Castelnuovo--Mumford) regularity.
We just recall here some basic facts about these topics and refer the interested reader to 
\cite[Chapter 2.C]{BV} for further details.

The regularity of a homogeneous ideal $I$ is a measure of its computational complexity. If $I$ is generated in a single degree $D$, having a \textit{linear resolution} means that $I$ has regularity equal to $D$. This is a desirable property for an ideal, since it means that the ideal is \lq \lq computationally simple\rq \rq. When $I$ is a homogeneous ideal generated in different degrees, the closest notion to linearity  is \textit{componentwise linearity}, meaning that for each degree $D$, the ideal generated by homogeneous elements of $I$ of degree $D$, i.e., $\langle I_D\rangle$, has a linear resolution
(see \cite{HH}). 
In this case, the regularity is equal to the maximum degree of the minimal generators of $I$.

\begin{corollary}
    The ideal $\mathcal{I}(\M^\mathrm{IG}_d)$ has a componentwise linear (minimal) free resolution given by the Eagon--Northcott complex. In particular, $\reg (\mathcal{I}(\M^\mathrm{IG}_d))=4$.
\end{corollary}

\begin{proof}
    This follows directly from  \cite[Proposition 4.5]{NR15} and the above discussion.
\end{proof}

\begin{corollary}\label{conj:degree_for_inverse_gaussian}
The degree is given by $\deg (\mathcal{I}(\M^\mathrm{IG}_d))= (d-1)^2$.
\end{corollary}

\begin{proof}
   Since the ideal is Cohen--Macaulay and of codimension $d-2$, the degree of the variety is the elementary symmetric polynomial of degree $d-2$ in $d$ unknowns, evaluated at $e_1=2,e_2=1,\ldots,e_d=1$ where $e_i$ is the degree of the $i$th column's entries. This yields $\deg (\mathcal{I}(\M^\mathrm{IG}_d))= (d-1)^2.$ The previous formula is known as Thom--Porteous--Giambelli formula. We refer the reader to \cite{FP} for further details on the degrees of determinantal varieties.
\end{proof}

It turns out that the determinantal realization of $\I(\M^\mathrm{IG}_d)$ from Theorem~\ref{thm:invgauss ideal} provides a Gröbner basis with respect to the reverse lexicographic ordering.

\begin{proposition}\label{prop:invgauss gb}
    The $3\times 3$ minors of $H_d^\mathrm{IG}$ form a Gröbner basis for $\I(\M^\mathrm{IG}_d)$ with respect to any antidiagonal term order (for example, the reverse lexicographic ordering).    Furthermore, the Hilbert series of $S/\I(\M^\mathrm{IG}_d)$ is given by
$$ \dfrac{1+(d-2)t+ \binom{d-1}{2}t^2+ \binom{d-1}{2}t^3}{(1-t)^3}.$$
\end{proposition}

\begin{proof}
    Let $I=\init(\I(\M^\mathrm{IG}_d))$, and let $J:=\left( \init([i,j,k]) \mid 1\leq i\leq j\leq k \leq d\right)$ where $[i,j,k]$ represents the 3-minor of $H_d^\mathrm{IG}$ on columns $i,j$ and $k$. Then $J$ is given by $$J= ( m_2,\ldots,m_{d-2})^3+ m_1^2 ( m_1,\ldots,m_{d-2})^2.$$
    We want to apply Lemma \ref{lemma: Grobner basis IG} to conclude that $J=I$.
    Note that $J$ is a primary ideal since every variable $m_i$ that divides a generator of $J$ appears with some power $m_i^{p_i}\in J$ \cite[Excercise 3.6]{Eis}, and $\sqrt{J}=(m_1,\ldots, m_{d-2})$. In particular, $J$ is unmixed and
    $$\dim (S/J)=\dim (S/\sqrt{J})=3= \dim (S/I).$$
    We are left to show that $\deg(J)=\deg (I)$. We already know that $\deg (I)=(d-1)^2$ by Corollary \ref{conj:degree_for_inverse_gaussian}. 
    As for $\deg(J)$, we will apply Lemma \ref{lemma:SES} to the ideal $J$ and a monomial generator $M$ of degree 4, i.e., $M \in \{m_1^4, m_1^3 m_2, m_1^2 m_2^2, \ldots, m_1^2 m_{d-2}^2 \}$. 
    
    Iteratively applying Lemma \ref{lemma:SES}, we get
    \begin{equation}\label{HS-SES}
        \HS_{S/J}(t)= \HS_{S/(m_2,\ldots,m_{d-2})^3}(t)- \binom{d-1}{2}\HS_{S/(m_2,\ldots,m_{d-2})[-4]}(t),
    \end{equation}
    where the factor $\binom{d-1}{2}$ appears in the previous identity because we are applying (\ref{SES}) to each of the $\binom{d-1}{2}$ degree 4 monomials arising from a 3-minor of $H_d^\mathrm{IG}$ involving the first column. Since $S/(m_2,\ldots,m_{d-2})\cong\CC[m_0,m_1,m_{d-1},m_d]$, we have
    \begin{equation}\label{HS1}
           \HS_{S/(m_2,\ldots,m_{d-2})[-4]}(t)=\frac{t^4}{(1-t)^4}. 
    \end{equation}
    As for $A:=S/(m_2, \ldots,m_{d-2})^3$, we can consider its quotient by the regular sequence given by $m_0,m_1,m_{d-1},m_d$, that is $$\overline{A}= \dfrac{S}{\left((m_2, \ldots,m_{d-2})^3+ (m_0,m_1,m_{d-1},m_d)\right)}.$$ 
    Since $S/(m_2, \ldots,m_{d-2})^3$ is Cohen--Macaulay of Krull dimension $4$, it follows that $\overline{A}$ is a 0-dimensional ring, usually called the \textit{Artinian reduction} of $A$. The Hilbert series of $\overline{A}$ is given by the Hilbert polynomial of $\CC[m_2,\ldots,m_{d-2}]/(m_2,\ldots,m_{d-2})^3$, which is 
    \begin{equation*}
        HS_{\overline{A}}(t)=1+(d-3)t+\binom{d-1}{2}t^2.
    \end{equation*}
     Since the numerator of the Hilbert series remains unchanged when taking the Artinian reduction (one can also see this by applying Lemma~\ref{lemma:SES} for each of the linear generators $m_0,m_1,m_{d-1},m_d$), we get 
    \begin{equation}\label{HS2}
        HS_{S/(m_2,\ldots,m_{d-2})^3}(t)=\dfrac{1+(d-3)t+\binom{d-1}{2}t^2}{(1-t)^4}. 
    \end{equation}
Hence, by identity (\ref{HS-SES}),
$$ \HS_{S/J}(t)=\dfrac{1+(d-3)t+\binom{d-1}{2}-\binom{d-1}{2} t^4}{(1-t)^4}=\dfrac{1+(d-2)t+ \binom{d-1}{2}t^2+ \binom{d-1}{2}t^3}{(1-t)^3}.$$
This yields $$\deg(J)= 1+(d-2)+ 2\binom{d-1}{2}=(d-1)^2=\deg (I).$$
Applying Lemma \ref{lemma: Grobner basis IG}, we get that $J=\init(\I(\M^\mathrm{IG}_d))$, which concludes the proof.

Since the Hilbert series does not change when taking the initial ideal, we have also found the Hilbert series of $\I(\M^{\mathrm{IG}}_d)$.
\end{proof}

\begin{example}
\label{ex:generators_IG}
For $d=3$, we have the following principal homogeneous prime ideal:
\begin{align*}
    \mathcal{I}(\M^\mathrm{IG}_3)&=(-m_{0}^{2} m_{1} m_{3}+3 m_{0}^{2} m_{2}^{2}-3 m_{0} m_{1}^{2} m_{2}+m_{1}^{4}
).
\end{align*}
Its Hilbert series is $(1+t+t^2+t^3)/(1-t)^3$, and its degree is $4$.

For $d=4$, we have the following  homogeneous prime ideal:
\begin{align*}
\mathcal{I}(\M^\mathrm{IG}_4)&=\left(\begin{array}{l}
-m_{0}^{2} m_{1} m_{3}+3 m_{0}^{2} m_{2}^{2}-3 m_{0} m_{1}^{2} m_{2}+m_{1}^{4},
\\
-3 m_{0}^{2} m_{2} m_{4}+5 m_{0}^{2} m_{3}^{2}-5 m_{1}^{3} m_{3}+3 m_{1}^{2} m_{2}^{2},
\\
-m_{0}^{2} m_{1} m_{4}+5 m_{0}^{2} m_{2} m_{3}-5 m_{0} m_{1}^{2} m_{3}+m_{1}^{3} m_{2},
\\
-3 m_{0} m_{2} m_{4}+5 m_{0} m_{3}^{2}+m_{1}^{2} m_{4}-6 m_{1} m_{2} m_{3}+3 m_{2}^{3}
\end{array}\right).
\end{align*}
Its Hilbert series is $(1+2t+3t^2+3t^3)/(1-t)^3$, and its degree is $9$.
\end{example}

We end our discussion about the algebraic properties of $\M^\mathrm{IG}_d$ by computing its singular locus.

\begin{proposition}
\label{prop:singular_locus_inverse_gaussian}
The singular locus of $\M^\mathrm{IG}_d$ is given by the line $m_0=m_1=\cdots=m_{d-2}=0$ and the point $m_1=m_2= \cdots=m_d=0$ in $\PP^d$.
\end{proposition}

\begin{proof}
We begin by noting that the open affine patch $\M^\mathrm{IG}_d\cap\{m_0m_1\neq 0\}$ is included in the smooth locus. To see this, note that the 3-minors involving the first two columns of $H_d^\mathrm{IG}$ together give rational expressions for $m_3,m_4,\ldots,m_d$ in the variables $m_0,m_1,m_2$, with denominators that are monomials in $m_0$ and $m_1$, which gives an isomorphism of varieties $(\CC^*)^2\times\CC\cong \M^\mathrm{IG}_d\cap\{m_0m_1\neq 0\}$.
The complement of this affine patch in $\M_d^\mathrm{IG}$ is the union of two projective lines:
$$\mathcal{L}_1:\:\:\: m_0=m_1=\cdots=m_{d-2}=0,\qquad \mathcal{L}_2:\:\:\: m_1=m_2=\cdots=m_{d-1}=0.$$
The singular locus of $\M^\mathrm{IG}_d$ must therefore be contained in $\mathcal{L}_1\cup\mathcal{L}_2$. 

Let $\mathcal{J}$ be the $\binom{d}{3}\times(d+1)$ Jacobian of the maximal minors of $H_d^\mathrm{IG}$. The singular locus is precisely the set of points where the rank of $\mathcal{J}$ is less than $\codim(\I(\M^\mathrm{IG}_d))=d-2$.

If we evaluate $\mathcal{J}$ at $\mathcal{L}_1$, we get a matrix with just $d-3$ nonzero entries. To see this, note that the only minors that can give a nonzero contribution to the Jacobian are those that have a term that is at most linear in the variables $m_0,\ldots,m_{d-2}$. The only such maximal minor arises from picking column indices $\{i,d-1,d\}$ for $i\in\{2,\ldots,d-2\}$, and its only contribution to the Jacobian will come from the term $(2d-3)m_{d-1}^2m_{i-2}$.

If, on the other hand, we evaluate $\mathcal{J}$ at $\mathcal{L}_2\setminus\mathcal{L}_1=\mathcal{L}_2\cap\{m_0\neq 0\}$, the only minors of $H_d^\mathrm{IG}$ that make a nonzero contribution to the Jacobian, are those that have total degree at most 1 in the variables $m_1,m_2,\ldots,m_{d-1}$. This corresponds to the column indices $\{1,i,d\}$ for $i\not\in\{1,d\}$ (which gives a minor containing the term $(2i-3)m_0^2m_im_d$) or $\{2,i,d\}$ for $i\not\in\{1,2,d\}$ (which gives a minor containing the term $(2i-3)m_0m_im_d$). From this, we see that $\mathcal{J}$ evaluated at $\mathcal{L}_2\cap\{m_0\neq 0\}$ has rank $d-2$ for $m_d\neq 0$, and vanishes completely for $m_d=0$. 
\end{proof} 

\begin{remark}
    Based on Corollary ~\ref{conj:degree_for_inverse_gaussian}, it might be tempting to believe that $\M^\mathrm{IG}_d$ is a Roman surface (a generic projection of a Veronese variety). However, the above proposition shows that this cannot be the case, since the singular locus of a Roman surface consists of 3 points, while the singular locus of $\M^\mathrm{IG}_d$ is given by a curve and a point.
\end{remark}

Before closing this section, we briefly turn our attention to the cumulants of the inverse Gaussian distribution. The cumulant generating function is 
\[K(t)=\log(M(t))=\frac{\lambda}{\mu}\left(1-\sqrt{1- \frac{2\mu^2t}{\lambda}}\right),\] from which we obtain the following formula for the cumulants (where $!!$ denotes the double factorial):
\[\kappa_r=\frac{(2r-3)!!\mu^{2r-1}}{\lambda^{r-1}}.\]
Similar to the moment variety, the ideal of the cumulant variety $\K^\mathrm{IG}_d\subseteq\CC^d$  can be realized as a determinantal ideal. 

\begin{proposition}
\label{prop:cumulant_variety_for_inverse_guassian}
The prime ideal $\mathcal{I}(\K^\mathrm{IG}_d)$ is generated by the $\binom{d-1}{2}$ quadrics given by the 2-minors of the matrix
\[K_d=\begin{pmatrix}
    -\kappa_1 & \kappa_2 & 3\kappa_3 & \cdots & (2d-3)\kappa_{d-1}\\
    \kappa_2 & \kappa_3 & \kappa_4 & \cdots & \kappa_d
\end{pmatrix}.\]
Furthermore, $\mathcal{I}(\K^\mathrm{IG}_d)$ is Cohen--Macaulay, its degree is $d-1$, and has a Gröbner basis given by the $2$-minors of $K_d$ with respect to any antidiagonal term order.
\end{proposition}

\begin{proof}
The recursive relation $\kappa_r=\frac{\mu^2}{\lambda}(2r-3)\kappa_{r-1}$ gives that $(\mu^2/\lambda,-1)\in\ker(K_d)$. We conclude that $J_d:=I_2(K_d)\subseteq\mathcal{I}(\K^\mathrm{IG}_d)$.
The matrix $H_d^\mathrm{IG}$ is a 1-generic Hankel matrix, so $J_d=I_2(K_d)$ is prime with the expected dimension $d-((d-1)-2+1)=2$ by Proposition~\ref{prop:Eis 1-gen}. In particular, it is Cohen--Macaulay. Since we have a parameterization $(\CC^*)\times\CC\to\K^\mathrm{IG}_d$ with algebraic identifiability, we also know that $\mathcal{I}(\K^\mathrm{IG}_d)$ is prime of dimension 2. We conclude that $\mathcal{I}(\K^\mathrm{IG}_d)=I_2(K_d)$. Similarly to $\mathcal{I}(\M^\mathrm{IG}_d)$, the degree of $\mathcal{I}(\K^\mathrm{IG}_d)$ can be computed via Thom--Porteous--Giambelli's formula. The Gr\"obner basis of determinantal ideals of symmetric matrices, and in particular Hankel matrices, was given by Conca (see \cite{Co1}, \cite{Co2}).
\end{proof}

Since the ideal obtained above is homogeneous, we note that the cumulant variety is a cone. 
In particular, it is a (scaled) rational normal curve when viewed as a projective curve in $\PP^{d-1}$.
An interesting question is whether any features of $\K^\mathrm{IG}_d$ can be pulled back by the birational moments-to-cumulants map $\M_d^\mathrm{IG}\dashrightarrow\K_d^\mathrm{IG}$. For example, we hope that this simple geometric model given by the cumulants could help prove non-defectiveness of secants of moment varieties, which corresponds in statistics to algebraic identifiability of mixtures of inverse Gaussian distributions (see \S \ref{subsec:identifiability} for further details).

\section{Gamma moment varieties}\label{sec:gamma}

In this section, we study moment varieties $\M_d^\Gamma$ for the gamma distribution. We find their defining ideals and use this to find their degrees and singular loci. We also discuss statistically important special cases of the gamma distribution whose moment varieties are rational normal curves embedded in $\M_d^\Gamma$.

\subsection{Gamma distribution}
The gamma distribution is commonly used to model physical and economic processes, especially in relation to arrival or waiting times. The density function is supported on $(0,+\infty)$, and involves the gamma function $\Gamma$ (see Figure \ref{fig:pdf}). The distribution has two parameters, and there are two commonly used parameterizations:
\begin{enumerate}
    \item The shape-scale parameterization is given by a shape parameter $k>0$ and a scale parameter $\theta>0$. The density function is given by
    \[f(x)=\frac{1}{\Gamma(k)\theta^k}x^{k-1}e^{-x/\theta}.\]
    \item The other parameterization is given by a shape parameter $\alpha>0$ and a rate parameter $\beta>0$. This parameterization is related to the previous one via $\alpha=k$ and $\beta=1/\theta$.
\end{enumerate}

The moment generating function is
$M(t)=(1-\theta t)^{-k}$, and the moments are given by
$$m_i=\theta^i\prod \limits_{j=0}^{i-1} (k+j)=\theta m_{i-1}(k+(i-1)).$$

Let $\M^\Gamma_d$ denote the $d$th moment variety of the gamma distribution. The affine part $\M^\Gamma_d\cap\{m_0\neq0\}$ has previously been studied in \cite[\S3.2]{aThesis}. Here, we study the full projective variety. We begin by finding its defining ideal. The following proof is similar to the proof of the defining ideal for the inverse Gaussian moment variety (Theorem~\ref{thm:invgauss ideal}), and so we omit some details.

\begin{theorem}\label{thm:gamma ideal}
    Let $d \geq 3$. The homogeneous prime ideal of the gamma moment variety $\M^\Gamma_d$ is  generated by the $\binom{d}{3}$ cubics given by the maximal minors of the following $(3\times d)$-matrix:
    $$H_d^\Gamma= 
\begin{pmatrix}
0  & m_1 & 2m_2 & 3m_3&\cdots& (d-1)m_{d-1}\\
m_0 & m_1 & m_2 & m_3 &\cdots& m_{d-1}\\
m_1 & m_2 & m_3 &m_4  & \cdots& m_d
\end{pmatrix}.
$$
Furthermore, $\M^\Gamma_d$ is Cohen--Macaulay.
\end{theorem}

\begin{proof}
Notice that the vector $(k\theta,\theta, -1)\neq 0$ is in the left kernel of the matrix $H_d^\Gamma$, so the 3-minors of $H_d^\Gamma$ vanish on $\M^{\Gamma}_d$. Thus, $$J_d:=I_2(H_d^\Gamma) \subseteq \mathcal{I}(\M^{\Gamma}_d).$$
Then, $\mathcal{V}(J_d) \supseteq \M^{\Gamma}_d$ and $\dim(\mathcal{V}(J_d)) \geq \dim(\M^{\Gamma}_d)=2$ as projective varieties.
On the other hand, if we fix an antidiagonal term order, then
 $$ \left(m_i m_j m_k \mid 0<i \leq j <k < d\right) \subseteq \init(J_d).$$
We observe that $ \sqrt{\left(m_i m_j m_k \mid 0<i \leq j <k < d\right)}= (m_i m_j \mid i \neq j, 0<i,j< d)$, so
$$\dim (\mathcal{V}(J_d)) =\dim (\mathcal{V}(\init(J_d)))\leq  \dim (\mathcal{V} \left(m_i m_j \mid i \neq j, 0<i,j< d\right))=2.$$
Thus, we have proved that $\dim (\mathcal{V}(J_d))= \dim (\M^{\Gamma}_d)=2$ as projective varieties.
This shows that the ring $R=S/J_d$ has expected Krull dimension $d+1-(d-3+1)=3$, so it is Cohen--Macaulay by Proposition \ref{prop:CMdetIdeals}.
In particular, the ideal $J_d$ has no embedded components by Proposition \ref{prop:CMproperties}.

We now show that $J_d$ is prime. Let $\Delta= m_0m_1$ be the $2$-minor in the upper-left corner (after multiplying by $-1$), and consider the localization $R_{\Delta}$. As in the inverse Gaussian case, it is sufficient to prove that $\Delta$ is a non-zero-divisor in $R$, and that $R_\Delta$ is a domain.

To prove that $\Delta$ is a non-zero-divisor in $R$, we prove that $m_0$ and $m_1$ are both non-zero-divisors in $R$. By Proposition \ref{prop:CMproperties}, since $R$ is Cohen--Macaulay, it suffices to prove that $\dim(R/(m_0))= \dim(R/(m_1))= \dim(R)-1= 2$. Observe that
$$R/(m_0)\cong \dfrac{\mathbb{C}[m_1,\ldots,m_d]}{I_3 ((H_d^\Gamma)\vert_{m_0=0})},$$
where $$(H_d^\Gamma)\vert_{m_0=0}= \begin{pmatrix}
0  & m_1 & 2m_2 & 3m_3&\cdots& (d-1)m_{d-1}\\
0 & m_1 & m_2 & m_3 &\cdots& m_{d-1}\\
m_1 & m_2 & m_3 &m_4  & \cdots& m_d
\end{pmatrix}.$$
If we consider the 3-minor on the first three columns, we get $m_1^2 m_2=0$. Thus we have two possibilities, either $m_1=0$ or $m_2=0$. Then, for $i=2,\ldots,d-2$, the $i$th antidiagonal is given by $(i+1)m_i^2m_{i+1}$, and we inductively see that either $m_i$ or $m_{i+1}$ must be equal to zero. This shows that $\mathcal{V}\left(I_3 (H_d^\Gamma) \right)\cap\{m_0=0\}$ is a union of $d-1$ curves given by
\[m_0=m_1=\cdots=\widehat{m_i}=\cdots=m_{d-1}=0,\qquad i=1,\ldots,d-1,\]
where the hat denotes omission. Thus, $R/(m_0)$ has affine dimension $2$. A similar argument shows that $R/(m_1)$ has affine dimension $1$. This shows that $\Delta$ is a non-zero-divisor.

The proof that $R_\Delta$ is a domain is similar to the inverse Gaussian case. One can use the $3$-minors of $H_d^\Gamma$ to see that $R_\Delta=\CC[\overline{m_0},\overline{m_1},\overline{m_2},\overline{1/m_0m_1}]\subseteq R$, and so $R_\Delta\cong\CC[w,x,y,z]/(wxz-1)$. Since $(wxz-1)$ is prime, this shows that $R_\Delta$ is a domain, which completes the proof.
\end{proof}

By Theorem \ref{thm:gamma ideal}, we have that $\mathcal{I}(M^\Gamma_d)$ is a Cohen--Macaulay ideal generated by the maximal minors of a matrix with linear entries. Therefore, $\mathcal{I}(M^\Gamma_d)$ has a linear minimal free resolution given by the Eagon--Northcott complex, and the degree of the gamma moment variety can be calculated via \cite[Proposition 2.15]{BV}.

\begin{corollary}\label{cor:gamma deg}
    The ideal $\mathcal{I}(\M^\Gamma_d)$ has a linear (minimal) free resolution given by the Eagon--Northcott complex. In particular, the regularity of the ideal is $\text{reg}(\mathcal{I}(\M^\Gamma_d))=3$ and the degree of the gamma moment surface is $\deg(\mathcal{I}(\M^\Gamma_d))=\binom{d}{2}$. 
\end{corollary}

We now show that the $3$-minors of $H_d^\Gamma$ form a Gröbner basis for $\mathcal{I}(\M^\Gamma_d)$ with respect to the reverse lexicographic order. The outline of the proof is as in the inverse Gaussian case. Let $J$ denote the ideal generated by the initial terms of the minors. We show in the following two lemmas that $J$ is unmixed, and that $\deg(J)=\binom{d}{2}$. We then apply Lemma~\ref{lemma: Grobner basis IG} to conclude that $J$ is indeed the initial ideal of $\mathcal{I}(\M^\Gamma_d)$.

To show that $J$ is unmixed, we apply the technique of polarization to obtain a squarefree monomial ideal (see \cite{Far06} for background on polarization). We then use Stanley--Reisner theory to study the associated primes of the polarization.

\begin{lemma}\label{lem:gamma unmixed initial}
    The ideal $J$ generated by initial terms of the $3$-minors of $H_d^\Gamma$ is unmixed.
\end{lemma}

\begin{proof}
    Define the ideal $J'$ of $\CC[m_1,\ldots,m_{d-1}]$ by eliminating $m_0$ and $m_d$:
    \[J'= J\cap \mathbb{C}[m_1,\ldots,m_{d-1}].\]
    If $J'$ is unmixed, then $J$ is as well. We show that every associated prime of $J'$ has height $d-2$. 
    Since $J'$ is a monomial ideal, its associated primes are also monomial. We therefore show that if $P\in\Ass(J')$, then $P$ is of the form
    \[P=(m_{i_1},\ldots,m_{i_{d-2}}\mid 1\leq i_1<\cdots<i_{d-2}\leq d-1).\]

    Let $\mathcal{P}(J')$ be the polarization of $J'$, so
    \begin{align*}
        \mathcal{P}(J')&=(m_{i,1}m_{i,2}m_{j,1},\,m_{i,1}m_{j,1}m_{k,1}\mid 1\leq i<j<k\leq d-1]\\
        &\subseteq S=\mathbb{C}[m_{i,1},m_{i,2},m_{d-1,1}\mid 1\leq i\leq d-2].
    \end{align*}
    Suppose $Q$ is an associated prime of $\mathcal{P}(J')$, and define the set $A\subseteq [d-1]$ to be
    \[A=\{i\mid m_{i,j}\in Q\}.\]
    By the correspondence between associated primes of an ideal and those of its polarization, to show that $J'$ is unmixed, we therefore want to show that $\#A=d-2$ \cite[Corollary~2.6]{Far06}.
    
    Let $\Delta$ denote the Stanley--Reisner complex associated to $J'$. Then a prime ideal $Q\subseteq S$ is an associated prime of $\mathcal{P}(J')$ if and only if the variables of $S$ that are not generators of $Q$ form a facet of $\Delta$. 

    First, suppose for contradiction that $\#A=d-1$. Let $F$ be the facet corresponding to $Q$. Then $F$ must only involve variables $m_{i,j}$ with $1\leq i\leq d-2$, and it cannot contain both $m_{i,1}$ and $m_{i,2}$ for any $i$. Since $F$ is a facet, the number of $i$ such that $m_{i,1}\in F$ cannot be greater than $2$; otherwise, this would correspond to the generator of $\mathcal{P}(J_d')$ given by $m_{i,1}m_{j,1}m_{k,1}$. We then have the following cases:
    \begin{enumerate}
        \item Suppose $m_{i,1},m_{j,1}\in F$ with $i<j$. This implies that 
        \[F=\{m_{i,1},m_{j,1},m_{k_1,2},\ldots,m_{k_{d-3},2}\mid k_\ell\in[d-2]\setminus\{i,j\}\}.\]
        Then, $F$ cannot be a facet, as $F\cup\{m_{j,2}\}$ is a face of $\Delta$.
        \item If there is at most one element of the form $m_{i,1}$ in $F$, then $F$ cannot be a facet, as it is contained in $F\cup\{m_{d-1,1}\}$, which is also a face of $\Delta$.
    \end{enumerate}

    Now suppose for contradiction that $\#A<d-2$, and let $F$ be the corresponding facet to $Q$. Then one of the following must be true:
    \begin{enumerate}
        \item $m_{i,1},m_{i,2},m_{d-1,1}\in F$ for some $1\leq i\leq d-2$, or
        \item $m_{i,1}m_{i,2},m_{j,1},m_{j,2}\in F$ for some $1\leq i<j\leq d-2$.
    \end{enumerate}
    In both cases, we obtain a contradiction that $F$ is a face of $\Delta$, since $\mathcal{P}(J')$ has generators of the form $m_{i,1}m_{i,2}m_{j,1}$ for all $1\leq i<j\leq d-1$.

    We therefore see that any associated prime of $J'$ has height $d-2$, and so $J$ is unmixed.
\end{proof}

We now prove that $J$ is indeed the initial ideal by computing the degree of $J$; the argument is similar to the proof of Proposition~\ref{prop:invgauss gb}. This result also addresses \cite [Conjecture 3.2.5]{aThesis}.

\begin{proposition}\label{prop:gamma_gb_hs}
    The $3$-minors of $H_d^\Gamma$ form a Gröbner basis for the homogeneous prime ideal of $\M_d^\Gamma$ with respect to any antidiagonal term order.     Moreover, the Hilbert series of $S/\mathcal{I}(\mathcal{M}^\Gamma_d)$ is given by
    \[\frac{1+(d-2)t+\binom{d-1}{2}t^2}{(1-t)^3}.\]
\end{proposition}

\begin{proof}
    Let $J$ be the ideal generated by initial terms of $3$-minors of $H_d^\Gamma$. We have that $J\subseteq\init(\mathcal{I}(\mathcal{M}^\Gamma_d))$, and Lemma~\ref{lem:gamma unmixed initial} shows that $J$ is unmixed. Therefore, to show equality of the ideals, it suffices by Lemma~\ref{lemma: Grobner basis IG} to show that the degrees of the two ideals are equal.
    
    We iteratively apply Lemma~\ref{lemma:SES}. Each time, we remove a minimal generator of $J$ of the form $m_i^2m_j$, with $1\leq i<j\leq d-1$; there are $\binom{d-1}{2}$ such generators. Then the colon ideal is always of the form $(m_{i_1},\ldots,m_{i_{d-3}})$, and the final $J'$ ideal that we end up with is
    \[J'=(m_im_jm_k\mid 1\leq i<j<k\leq d-1).\]
    Thus, the Hilbert series of $J$ is given by
    \begin{align*}
        \text{HS}_{S/J}(t)&=\text{HS}_{S/J'}(t)-\binom{d-1}{2}t^3\text{HS}_{S/(m_1,\ldots,m_{d-3})}(t)\\
        &=\frac{1+(d-3)t+\binom{d-2}{2}t^2}{(1-t)^4}-\frac{\binom{d-1}{2}t^3}{(1-t)^4}\\
        &=\frac{1+(d-2)t+\binom{d-1}{2}t^2}{(1-t)^3}.
    \end{align*}
    Substituting $t=1$ in the numerator, we see that
    \[\deg(J)=1+(d-2)+\binom{d-1}{2}=\binom{d}{2}.\]
    By Corollary~\ref{cor:gamma deg}, this is also the degree of $\mathcal{I}(\mathcal{M}^\Gamma_d)$, and so this completes the proof.
\end{proof}

\begin{example}
\label{ex:generators_gamma}
For $d=3$, we have the following principal homogeneous prime ideal:
\begin{align*}
    \I(\M_3^\Gamma)=(-m_0 m_1 m_3 + 2 m_0 m_2^2 - m_1^2 m_2).
\end{align*}
Its Hilbert series is $(1+t+t^{2})/(1-t)^{3}$, and its degree is $3$.

For $d=4$, we have the following  homogeneous prime ideal:
\begin{align*}
\I(\M_4^\Gamma)=\begin{pmatrix}
    -m_0 m_1 m_3 + 2 m_0 m_2^2 - m_1^2 m_2,\\ -m_0 m_1 m_4 + 3 m_0 m_2 m_3 - 2 m_1^2 m_3,\\ -2 m_0 m_2 m_4 + 3 m_0 m_3^2 - m_1 m_2 m_3,\\ -m_1 m_2 m_4 + 2 m_1 m_3^2 - m_2^2 m_3
\end{pmatrix}.
\end{align*}
Its Hilbert series is $(1+2t+3t^{2})/(1-t)^3$, and its degree is $6$.
\end{example}

We now describe the singular locus of $ \M^\Gamma_d$.

\begin{proposition}
\label{prop:singular_locus_gamma}
The singular locus of $\M^\Gamma_d$ is given by two points in $\PP^d$, defined by the ideals $(m_0,m_1,\ldots,m_{d-1})$ and $(m_1,m_2,\ldots,m_d)$.
\end{proposition}

\begin{proof}
Analogously to the proof of Proposition~\ref{prop:singular_locus_inverse_gaussian}, one can construct an isomorphism $(\CC^*)^2\times\CC\cong\M^{\Gamma}_d\cap\{m_0m_1\neq 0\}$ by considering the maximal minors involving the first two columns of $H_d^\Gamma$. The complement of this affine open patch in $\M_d^\Gamma$ is the union of $2(d-1)$ projective curves:
\begin{itemize}
    \item $m_0=m_1=\cdots \widehat{m_{i}} =\cdots = m_{d-2}=m_{d-1}=0$ for $i\in\{1,\ldots,d-1\}$, and
    \item $m_1=m_2=\cdots \widehat{m_{i}} =\cdots = m_{d-1}=m_{d}=0$ for $i\in\{2,\ldots,d\}$.
\end{itemize}
Let $\mathcal{J}$ denote the Jacobian of the maximal minors of $H_d^\Gamma$. If we evaluate $\mathcal{J}$ at one of the curves of the first kind described above, then the only minors that make a contribution to $\mathcal{J}$ are those that contain a term that is at most linear in $m_0,\ldots,\widehat{m_i},\ldots,m_{d-1}$. This happens precisely for column indices $\{i,i+1,j\}$ for $j\not\in\{i,i+1\}$, which gives a minor with a term $im_i^2m_{j+1}$, and $\{i+1,j,d\}$ for $j\not\in\{i+1,d\}$, which gives a minor with a term divisible by $m_im_jm_d$. Hence, the resulting evaluated matrix $\mathcal{J}$ vanishes completely when $m_i=0$, and has rank at least $d-2$ if $m_i\neq 0$. The situation is analogous when $\mathcal{J}$ is evaluated on a curve of the second kind.
\end{proof}

We conclude this section by pointing out that, similar to the cumulant variety of the inverse Gaussian distribution, the cumulant variety of the gamma distribution is a cone. This result first appeared in \cite [Proposition 3.2.1]{aThesis}. We include it here for the sake of completeness.

\begin{proposition}
\label{prop:cumulant_variety_for_gamma}
The ideal $\mathcal{I}(\K^\Gamma_d)$ is generated by the $\binom{d-1}{2}$ quadrics given by the 2-minors of the matrix
\[K_d=\begin{pmatrix}
    \kappa_1 & \kappa_2 & \dfrac{\kappa_3}{2} & \cdots & \dfrac{\kappa_{d-1}}{(d-2)!}\\
    \kappa_2 & \dfrac{\kappa_3}{2} & \dfrac{\kappa_4}{3!} & \cdots & \dfrac{\kappa_d}{(d-1)!}
\end{pmatrix}.\]
Furthermore, $\mathcal{I}(\K^\Gamma_d)$ is Cohen--Macaulay, its degree is $d-1$, and its Gröbner basis is given by the $2$-minors that generate the ideal.
\end{proposition}

Similar to the case of the inverse Gaussian distribution, having this simpler geometric model given by the cumulants could help in understanding the geometry of secants of moment varieties of the gamma distribution. 

\subsection{Exponential and chi-squared distributions}\label{subsec:1dim}
We now consider two important special cases of the gamma distribution: the exponential and chi-squared distributions. Both distributions are given by one parameter, and have one-dimensional moment varieties. 
We show that the moment varieties are in fact rational normal curves. 

The \emph{exponential distribution} is the distribution of the time between events in a process, where events occur continuously and independently at a constant average rate $\lambda>0$. The only parameter is $\lambda$, and the moment generating function is
\[M(t)=(1-\lambda t)^{-1}.\]
The moments are given by 
\[m_i=i! \lambda^i= i \lambda m_{i-1}.\] 
This is a specialization of the gamma distribution with shape $k=1$ and rate $\theta=\lambda$ under the shape-scale parameterization.

Let $\M^{\exp}_d$ denote the $d\text{th}$ moment variety of the exponential distribution. The following result shows that the variety is a rational normal curve.

\begin{proposition}\label{prop:generators_for_exponetial}
The homogeneous ideal $\I(\M^{\exp}_d)$ is  generated by the maximal minors of
$$H_d^{\exp}= 
\begin{pmatrix}
m_0 & 2m_1 & 3m_2 & 4m_3&\cdots& dm_{d-1}\\
m_1 & m_2 & m_3 & m_4 &\cdots& m_{d}
\end{pmatrix}.$$
In particular, $\M^{\exp}_d$ is a rational normal curve.
\end{proposition}

\begin{proof}
Notice that the vector $(\lambda,-1)\neq 0$ is in $\ker (H_d^{\exp})$, so the 2-minors of $H_d^{\exp}$ vanish on $\M^{\exp}_d$. Thus, 
$$J_d:=I_2(H_d^{\exp}) \subseteq \mathcal{I}(\M^{\exp}_d).$$
Moreover, $H_d^{\exp}$ is a 1-generic Hankel matrix, so $J_d=I_2(H_d^{\exp})$ is prime of expected codimension $d-1$ by Proposition~\ref{prop:Eis 1-gen}. In particular, it defines a Cohen--Macaulay ring of dimension $2$. As projective varieties,
$$\dim (\mathcal{V}(J_d))= 1= \dim (\M^{\exp}_d).$$
Hence, $\mathcal{V}(J_d)=\M^{\exp}_d$. Since $J_d$ is prime, we get $J_d= \mathcal{I}(\M^{\exp}_d)$.
\end{proof}

We now consider the \emph{chi-squared distribution}, which is the distribution of a sum of the squares of $k\geq 1$ independent standard Gaussian variables. The only parameter is $k$, and the moment generating function is
\[M(t)=(1-2t)^{-k/2}.\]
The moments are given by 
\[m_i=k(k+2)\cdots (k+2i-2)=m_{i-1} (k+2i-2)=m_{i-1} k + m_{i-1} (2i-2).\]
This is a specialization of the gamma distribution with shape $k/2$ and scale $2$.

Let $\mathcal{M}^{\chi}_{d}$ denote the moment variety. Following the proof of Proposition~\ref{prop:generators_for_exponetial}, we attempt to find generators of the ideal $\I(\mathcal{M}^{\chi}_{d})$ by using the recursive formula given above, which gives that the maximal minors of
$$H_d^{\chi}= 
\begin{pmatrix}
0  & 2m_1 & 4m_2 & 6m_3&\cdots& (2d-2)m_{d-1}\\
m_0 & m_1 & m_2 & m_3 &\cdots& m_{d-1}\\
m_1 & m_2 & m_3 &m_4  & \cdots& m_d
\end{pmatrix}
$$
vanish on $\mathcal{M}^{\chi}_{d}$.
This matrix is similar to the one for the gamma distribution, and the dimension of the projective variety $\mathcal{V}(I_3(H_d^{\chi}))$ is $(d+1)-(d-3+1)-1=2$. However, the dimension of the projective variety $\mathcal{M}^{\chi}_{d}$ is $1$. Thus, the $3$-minors of $H_d^{\chi}$ are not sufficient to generate the ideal of $\mathcal{M}^{\chi}_{d}$.

Instead, we follow the technique used in \cite[\S3.3]{aThesis}. Here, the author uses a change in coordinates from moments to powers of the parameter to obtain the defining ideal of the Poisson moment variety. Let $s_{d,i}$ denote the (signed) Stirling numbers of the first kind, and $S_{d,i}$ denote the Stirling numbers of the second kind. Recall that the Stirling numbers of the first kind satisfy 
\[x(x-1)(x-2)\cdots(x-n+1)=\sum_{k=0}^n s_{n,k}x^k,\]
and that they are related to the Stirling numbers of the second kind by the following property. Let $(f_n)_{n=0}^\infty$ and $(g_n)_{n=0}^\infty$ be two sequences of complex numbers; then
\[g_d=\sum_{i=0}^d S_{d,i} f_i=\sum_{i=1}^d S_{d,i} f_i \quad \Leftrightarrow \quad f_d=\sum_{i=0}^d s_{d,i} g_i=\sum_{i=1}^d s_{d,i}g_i.\]
Note that $s_{n,0}=S_{n,0}=0$ for all $n\geq 1$.
Using the formula for the $d$th moment, we see that
\[m_d=\sum_{i=0}^d (-2)^{d-i}s_{d,i}k^i,\]
and so applying the relationship between Stirling numbers of the first and second kind with $f_d=m_d/(-2)^d$ and $g_d=(-k/2)^d$, we have that
\[k^d=\sum_{i=0}^d (-2)^{d-i}S_{d,i}m_i.\]

\begin{proposition}\label{prop:ideal chi-square}
    Let $d\geq 2$. The homogeneous prime ideal of the chi-squared moment variety $\mathcal{M}^{\chi}_{d}$ is  generated by the $2\times 2$ minors of the $(2\times d)$-matrix
    \[H_d^{\chi}=\begin{pmatrix}
m_0 & m_1 & m_2-2m_1 & \cdots & \sum_{i=0}^{d-1} (-2)^{d-i-1}S_{d-1,i}m_i\\
m_1 & m_2-2m_1 & m_3 - 6m_2 + 4m_1 & \cdots & \sum_{i=0}^d (-2)^{d-i}S_{d,i}m_i
    \end{pmatrix}.\]
    In particular, $\mathcal{M}^{\chi}_{d}$ is a rational normal curve.
\end{proposition}

\begin{proof}
    Consider the algebra homomorphism $\varphi\from\CC[m_0,\ldots,m_d]\to\CC[x_0,\ldots,x_d]$ defined by the linear map
    \[m_j\mapsto\begin{cases}
        \sum_{i=0}^j (-2)^{j-i}s_{j,i}x_i&\quad j\geq 1,\\
        x_0&\quad j=0.
    \end{cases}\]
    This linear map can be given by an upper triangular matrix and is invertible, with inverse
    \[x_j\mapsto\begin{cases}
\sum_{i=0}^j (-2)^{j-i}S_{j,i}m_i & \quad j\geq 1,\\
m_0&\quad j=0.
    \end{cases}\]
    Thus, $\varphi$ is an algebra isomorphism, and induces an isomorphism of varieties $\Phi\from\mathbb{P}^d\to\mathbb{P}^d$, where the coordinates of the domain are given by the $x_i$'s, and the coordinates of the codomain are given by the $m_i$'s.
    
    Under this map, a point $[1:m_1:\cdots:m_d]\in\mathcal{M}^{\chi}_{d}\cap\{m_0=1\}$ is mapped to $[1:x_1:x_1^2:\cdots:x_1^d]$. Thus, the homogenous prime ideal of the projective variety $\Phi(\mathcal{M}^{\chi}_{d})$ is  generated by the $2\times 2$ minors of the matrix
    \[H'_d=\begin{pmatrix} x_0 & x_1 & x_2 & \cdots & x_{d-1} \\
    x_1 & x_2 & x_3 & \cdots & x_d\end{pmatrix},\]
    and we conclude that the homogeneous prime ideal of $\mathcal{M}^{\chi}_{d}$ is given by $\varphi^{-1}(I_2(H'_d))=I_2(H_d^{\chi})$.
\end{proof}

An interesting direction for future work is to classify algebraic and geometric properties of moment varieties for polynomial distributions. For example, in \cite{aThesis}, the author studies moment varieties for the Poisson distribution, which is also a one-parameter distribution. These varieties were shown to be rational normal curves. Given our results on the moment varieties for the exponential and chi-squared distributions, this raises the following problem.

\begin{problem}
    Classify projective curves that can appear as moment varieties of one-parameter distributions.
\end{problem}

It is also worth noticing that all moment varieties studied in this paper are Cohen--Macaulay. Moreover, to the best of our knowledge, there is no known distribution for which the moment varieties fail to have this property (for example, those studied in \cite{AFS16} and \cite{aThesis} are all Cohen--Macaulay). This leads to the following problem.

\begin{problem}\label{Q:CM}
    Classify all Cohen--Macaulay moment varieties from polynomial distributions.
\end{problem}

\section{Future directions}\label{sec:future_directions}
In this section, we discuss two possible directions for the future algebraic study of moment varieties arising from inverse Gaussian and gamma distributions: identifiability of mixtures, and the complexity of the optimization problem that arises in the overdetermined scenario.

\subsection{Algebraic identifiability of mixtures}\label{subsec:identifiability}

It is often the case that real-world datasets exhibit multimodal behavior or heterogeneity. To model these datasets, statisticians use \textit{mixtures of distributions} which, thanks to their inherent flexibility, allow modeling complex and diverse patterns within data. Mixtures of distributions are simply convex combinations of probability distributions (see Figure~\ref{fig:mixture} for an example).

\begin{figure}[h]
\centering
\begin{tikzpicture}[
scale=0.8,
declare function={pdf_ig(\x,\mu,\lambda)=sqrt(\lambda/(2*pi*\x^3)*exp(-\lambda*(\x-\mu)^2/(2*\mu^2*\x));}
]
\begin{axis}[
axis lines=left,
enlargelimits=upper,
ymax=1,
samples=100,
column sep=0.2em,
row sep=0.2em,
legend cell align={left},
nodes={scale=0.88, transform shape},
legend entries={
Component with $\mu_1=1${,} $\lambda_1=5$,
Component with $\mu_2=2${,} $\lambda_2=20$, 
Mixture with $\alpha_1=0.4${,}
$\alpha_2=0.6$}
]
\addplot [smooth, domain=0:5, red, dashed]{0.4*pdf_ig(x,1,5)};
\addplot [smooth, domain=0:5, blue, dashed]{0.6*pdf_ig(x,2,20)};
\addplot [smooth, domain=0:5, purple]{0.4*pdf_ig(x,1,5)+0.6*pdf_ig(x,2,20)};
\end{axis}
\end{tikzpicture}
\caption{Probability density for a mixture of two inverse Gaussians.
}
\label{fig:mixture}
\end{figure}
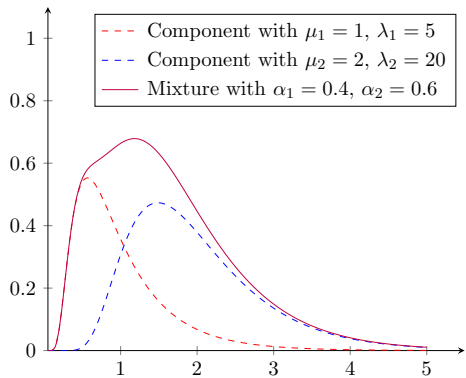

Geometrically, the moment variety of a distribution that is the mixture of $k$ independent observations of a distribution with moment variety $\M_d$ corresponds to the $k$th \emph{secant variety} of $\M_d$, which we denote by $\Sec_k(\M_d)$ (see, e.g., \cite{BCCGO18} for an overview on secant varieties). If the moments depend on $n$ parameters $\theta=(\theta_1,\ldots,\theta_n)$, it holds that $\Sec_k(\M_d)$ is given by the projective closure in $\PP^d$ of the image of the map
\begin{equation}\label{eq:parametrization_secant_variety}
\begin{array}{l}(\CC^n)^k\times\CC^{k-1}\to\CC^d\\\left((\theta^{(1)},\ldots,\theta^{(k)}),\alpha\right)\mapsto \left(\textstyle\sum_{i=1}^{k-1}\alpha_i\,m_r(\theta^{(i)})+(1-\textstyle\sum_{i=1}^{k-1}\alpha_i)\,m_r(\theta^{(k)})\right)_{r=1,\ldots,d}\,.\end{array}
\end{equation}
The dimension of  $\Sec_k(\M_d)$ is  bounded above by the dimension of the domain of \eqref{eq:parametrization_secant_variety}, so we have 
$$ \dim \left(\Sec_k (\M_d)\right) \leq \min\{d, nk+k-1\}.$$
If equality holds we say that $\Sec_k(\M_d)$ is \emph{nondefective}; otherwise, we say that $\M_d$ is \emph{$k$-defective}. In particular, if $\Sec_k(\M_d)$ is nondefective and $d\geq nk+k-1$, then the parameters of each component, as well as the mixing coefficients, are algebraically identifiable from the $nk+k-1$ first moments. This motivates the study of \textit{non-defectiveness} of $\Sec_k(\M_d)$. 
It was proven in \cite{ARS18} that $\Sec_k(\M_d)$ is nondefective for all $k,d\geq 2$ in the Gaussian case, and a natural question is whether the same is true for the distributions considered in this paper.

For the exponential and chi-squared distributions, non-defectiveness (and, in fact, rational identifiability) follows directly from the determinantal realizations found in \S\ref{subsec:1dim}.

\begin{proposition}
For $k$-mixtures of the exponential distribution or the chi-squared distribution, we have rational identifiability from the first $2k-1$ moments.
\end{proposition}

\begin{proof}
Propositions~\ref{prop:generators_for_exponetial} and \ref{prop:ideal chi-square} give that both $\M_d^{\exp}$ and $\M_d^{\chi}$ are linearly isomorphic to the rational normal curve in $\PP^d$. The result now follows from the fact that a general point on the $k$th secant variety of the rational normal curve in $\PP^d$ lies on a unique $k$-secant if $2k\leq d+1$ (since any $d+1$ points on such a curve are linearly independent \cite[Example~1.14]{Har13}). 
\end{proof}

The cases of the inverse Gaussian and gamma distributions are more complicated. An elementary first observation one can make is that the dimension of $\Sec_k(\M_d)$ is bounded below by the rank of the Jacobian of the parametrization \eqref{eq:parametrization_secant_variety} at any point in the domain. By computing this rank for randomly chosen points with rational entries in exact arithmetic in \texttt{Maple}, we are able to computationally verify non-defectiveness for all $d,k\leqslant 100$. Based on this, we conjecture the following.

\begin{conjecture}\label{conj:nondefec}
  $\Sec_k (\M^\mathrm{IG}_d)$ and $\Sec_k (\M^{\Gamma}_d)$ are  nondefective for all $d\geq 2$ and $k\geq 2$.
\end{conjecture}
    
A helpful step towards a proof of this conjecture is
Terracini's classification of $k$-defective surfaces (see, e.g., \cite[Theorem~8]{ARS18} and \cite[Theorem~1.3]{CC02}), which establishes that the only $k$-defective surfaces are either cones or quadratic Veronese embeddings of a rational normal surface in $\PP^k$. Similar to in \cite{ARS18}, we can rule out the latter possibility.

\begin{proposition} \label{prop:contained_in_cone}
    Let $\M_d$ be either $\M^\mathrm{IG}_d$ or $\M^{\Gamma}_d$. If $\M_d$ is $k$-defective, then it is contained in a cone over a curve, with apex a linear space of dimension at most $ k-2$. 
\end{proposition}

\begin{proof}
Assume for a contradiction that $\M_d$ is not contained in such a cone. Then, by Terracini's classification of defective surfaces, it should be the quadratic Veronese embedding of a rational normal surface in $\PP^k$, and in particular have at most one singular point. But this is a contradiction, since the singular locus of $\M_d$ contains a line in the case of the inverse Gaussian distribution by Proposition~\ref{prop:singular_locus_inverse_gaussian}, and  two points in the case of the gamma distribution by Proposition~\ref{prop:singular_locus_gamma}.
\end{proof}

Ruling out the possibility of $\M_d$ being contained in a cone of the type described in Proposition~\ref{prop:contained_in_cone} is ongoing work with Kristian Ranestad.

\subsection{Identifiability degree}\label{subsec: ID}
Suppose Conjecture~\ref{conj:nondefec} is true, so that $\dim\!\left(\Sec_k(\M_d)\right)$ equals $\min\{3k-1,d\}$ for both the inverse Gaussian and gamma distribution. Then, for $d=3k-1$, we have \emph{generically finite fibers} of the parameterization \eqref{eq:parametrization_secant_variety}.
The generic cardinality of these fibers is the \emph{identifiability degree}, which measures the complexity of solving the moment equations. Identifiability degrees for mixtures of Gaussians have previously been studied with numerical algebraic geometry tools \cite{ALR16}, and here we take a similar approach.

Table~\ref{table:ID} contains numerical bounds on the identifiability degree up to the label-swapping symmetry for $d=3k-1$. The numbers were computed using monodromy in \texttt{HomotopyContinuation.jl} (with 10 loops without new solutions as the stopping condition), and certified as lower bounds with the techniques of \cite{BRT23}. When numerically feasible, the Pseudo-Segre trace test from \cite[\S3.3]{ALR16} was used to verify completeness of the solution set, using a tolerance of $\varepsilon=10^{-12}$.

\begin{table}[h]
    \centering
    \begin{tabular}{cccc}
    \toprule
        $k$ & 2 & 3  & 4 \\
        \midrule
        Gaussian 
        & 9 & 225 & $\geq 10\,350$*\\
        Gamma & $\geq 9$*  & $\geq 242$* & $\geq 13\,327$ \\
        Inverse Gaussian & $\geq 24$*  & $\geq 1637$ & $\geq 20\,000$ \\
        \bottomrule
    \end{tabular}
    \caption{Identifiability degrees for $k$-mixtures of three distributions, up to label-swapping symmetry. An asterisk is used to mark bounds for which a trace test indicates sharpness with high probability. The monodromy calculation for $k=4$ for the inverse Gaussian was manually terminated, and is therefore not expected to be sharp. The three values for the Gaussian distribution were computed in the works \cite{Pea94}, \cite{AFS16}, and \cite{ALR16}, respectively.}
    \label{table:ID}
\end{table}

Many open questions in this direction remain. Apart from computing the very large identifiability degree for $k=4$ in the inverse Gaussian case, one can also investigate the case where some parameters are fixed, or impose equality of some of the parameters of the components of the mixtures. This has been a successful line of research in the Gaussian case \cite{LAR21,AAR21}. 
Another possible direction is to better understand the smaller identifiability degrees already found in Table~\ref{table:ID} from the point of view of \cite{Pea94}. For instance, we propose the following problem.

\begin{problem}
Find explicit univariate polynomials of degree 9 and 24, analogous to the ``Pearson polynomial'' found in \cite{Pea94}, that account for the identifiability degrees for $k=2$ in Table~\ref{table:ID}, e.g., using a similar computer-algebra strategy as in \cite[\S3]{AFS16}.
\end{problem}

Finally, we also note that numerical experiments for $k\leq 4$ indicate that a single additional polynomial is enough to decrease the generic number of solutions given in Table~\ref{table:ID} to one. This would be in line with \cite[Conjecture~15]{AFS16} in the Gaussian case, and we therefore conjecture the following. 

\begin{conjecture}
For $k$-mixtures of the inverse Gaussian distribution or the gamma distribution, we have rational identifiability from the first $3k$ moments.
\end{conjecture}

A natural first step could be to prove the more modest claim of rational identifiability from $3k+2$ moments, for instance by attempting to adapt the techniques previously employed to prove this in the Gaussian setting \cite[Theorem~4.4]{LAR21}.

\subsection{Euclidean distance degree}\label{subsec: EDD}

For noisy sample moments, the moment equations might not have any statistically meaningful solutions. In that case, a common approach is to instead solve the optimization problem
\begin{equation}\label{eq:optimization_problem}
    \min_{\theta\in\CC^n} \Vert m(\theta)-\widetilde{m}\Vert^2,
\end{equation}
where $m(\theta)=(m_1(\theta),\ldots,m_d(\theta))$ is the first $d$ moments as functions of $\theta$, $\widetilde{m}\in\RR^d$ is the first $d$ sample moments, and $d>n$ is such that $\dim(\M_d)=n$. Here, $\Vert{\cdot}\Vert$ denotes the Euclidean norm, but we remark that the computations done in this section can be adapted also to weighted norms $\Vert{\cdot}\Vert_W$ with $\Vert x\Vert_W^2=x^\top Wx$ for some (positive semidefinite) matrix $W\in\RR^{d\times d}$. Such norms are often considered in the more general framework of the \emph{generalized method of moments} \cite{Han82}.

Compared to maximum likelihood estimation (MLE), which instead minimizes the Kullback--Liebler divergence, the distance-based approach studied here has the advantage of being a \emph{polynomial} optimization problem. Hence, it has a finite \emph{optimization degree}, which is the number of complex critical points for generic sample moments $\widetilde{m}_i$. Understanding this degree is a difficult problem, but can be studied geometrically through the lens of the related notion of the \emph{Euclidean distance degree} (ED degree). 

The ED degree of $\M_d$ is the generic number of critical points of the optimization problem
\begin{equation}\label{eq:ED_problem}
\min_{m\in\M_d\cap\{m_0=1\}}\Vert m-\widetilde{m}\Vert^2.
\end{equation}
This notion was introduced in \cite{DHOST16}, and has been shown to have a rich geometric meaning in terms of polar degrees \cite[Section~5]{DHOST16} or Euler characteristics \cite{MRW20}. This might make the ED degree more approachable to compute than the optimization degree of \eqref{eq:optimization_problem}. Since the optimization degree is bounded below by the product of the identifiability degree and the ED degree of $\M_d$, understanding the ED degree is a natural subproblem.

In this work, we take a first step in this direction by computing the ED degree of $\M_3$ for the inverse Gaussian and gamma distribution. For comparison purposes, we also do this for the Gaussian distribution, which to our knowledge has not been studied from the ED degree perspective before. In all three cases, the affine patch $\{m_0=1\}$ is a hypersurface cut out by a single polynomial $f\in\CC[m_1,m_2,m_3]$ is given in Examples~\ref{ex:generators_IG} and \ref{ex:generators_gamma}, and \cite[Proposition~2]{AFS16}, respectively. The ED degree is the number of complex roots of the system
\begin{equation}
\label{eq:ED_system_hypersurface}
f(m)=k\,\nabla f(m)-(m-\widetilde{m})=0,\quad m\in\CC^3,\quad k\in\CC
\end{equation}
for generic parameters $\widetilde{m}\in\CC^3$. A standard Gröbner basis calculation in \texttt{Oscar.jl} over the field $\CC(\widetilde{m}_1,\widetilde{m}_2,\widetilde{m}_3)$ of rational functions in the sample moments proves the following result.

\begin{proposition}
\label{prop:EDD_hypersurfaces}
For the inverse Gaussian distribution, the gamma distribution, and the Gaussian distribution, the ED degree is given by
\[\operatorname{EDdegree}(\M_3^\mathrm{IG})=12\,,\quad \operatorname{EDdegree}(\M_3^\Gamma)=10,\quad \operatorname{EDdegree}(\M_3^\mathrm{Gaussian})=7. \]
\end{proposition}

This means that in all three cases, $f$ is generic enough to satisfy the mixed volume bounds given in  \cite[Theorem~1]{BSW22}, but not generic enough to satisfy the bound of \cite[Proposition~2.6]{DHOST16}, which evaluates to 52, 21, and 21, respectively. The determinantal realizations of the three hypersurfaces are also not generic enough in the family of $3\times 3$ matrices to give the value $1$ obtained from \cite[Example~2.3]{DHOST16}.

A natural future direction is to compute and analyze the ED degrees for higher values of $d$, e.g., using numerical algebraic geometry and ideas from the emerging field of \emph{metric algebraic geometry}. It would also be interesting to understand the ED degree of $\Sec_k(\M_d)$; since the implicitization problem for these secant varieties turns out to be very hard in our experience, this is a considerable challenge, even with numerical techniques.

% Extra space between items in the bibliography
\newlength{\bibitemsep}\setlength{\bibitemsep}{.2\baselineskip plus .05\baselineskip minus .05\baselineskip}
\newlength{\bibparskip}\setlength{\bibparskip}{0pt}
\let\oldthebibliography\thebibliography
\renewcommand\thebibliography[1]{
  \oldthebibliography{#1}
  \setlength{\parskip}{\bibitemsep}
  \setlength{\itemsep}{\bibparskip}
}

\bibliographystyle{alpha}
\bibliography{ref}

\vspace{1em}

\noindent {\bf Authors' addresses:}

\noindent 
Oskar Henriksson, University of Copenhagen \hfill{\tt oskar.henriksson@math.ku.dk}\\
Lisa Seccia, University of Neuchâtel \hfill{\tt lisa.seccia@unine.ch}\\
Teresa Yu, University of Michigan \hfill {\tt twyu@umich.edu}

\end{document}